\documentclass[12pt]{amsart}

\usepackage{amssymb,amsbsy}
\usepackage{latexsym,array}
\usepackage{verbatim,color}
\usepackage{fullpage,euscript}

\usepackage[all]{xy}
\usepackage{xcolor}
\usepackage[colorlinks=true,linkcolor=blue,urlcolor=my_color,citecolor=magenta]{hyperref}

\definecolor{my_color}{rgb}{0,0.5,0.5}
\definecolor{MIXT}{rgb}{0.8,0.5,0.2}

\numberwithin{equation}{section}

\input {cyracc.def}
\font\tencyr=wncyr10 
\font\tencyi=wncyi10 
\font\tencysc=wncysc10 
\def\rus{\tencyr\cyracc}

\def\rusi{\tencyi\cyracc}
\def\rusc{\tencysc\cyracc}

\newcounter{rmke}
\numberwithin{rmke}{section}

\newtheorem{thm}{Theorem}[section]
\newtheorem{lm}[thm]{Lemma}
\newtheorem{cl}[thm]{Corollary}
\newtheorem{prop}[thm]{Proposition}

\theoremstyle{remark}
\newtheorem{rmk}[thm]{Remark}

\theoremstyle{definition}
\newtheorem{ex}[thm]{Example} 
\newtheorem{opr}{Definition}
\newtheorem{Not}[thm]{Notation}
\newtheorem*{rema}{Remark}
\newtheorem{rme}[rmke]{Remark}

\theoremstyle{plain}
\newtheorem{thme}[rmke]{Theorem}
\newtheorem{lme}[rmke]{Lemma}

\newcommand {\g}{{\mathfrak g}}

\newcommand {\h}{{\mathfrak h}}

\newcommand {\q}{{\mathfrak q}}

\newcommand {\te}{{\mathfrak t}}

\newcommand {\gln}{{\mathfrak{gl}}_n}
\newcommand {\sln}{{\mathfrak{sl}}_n}

\newcommand {\eus}{\EuScript}

\newcommand {\gA}{{\eus A}}

\newcommand {\gI}{{\eus I}}
\newcommand {\gJ}{{\eus J}}

\newcommand {\gS}{{\eus S }}

\newcommand {\ap}{\alpha}

\newcommand {\lb}{\lambda}
\newcommand {\vp}{\varphi}

\newcommand {\N}{{\mathcal N}}

\newcommand {\BT}{{\mathbb T}}
\newcommand {\BV}{{\mathbb V}}
\newcommand {\BU}{{\mathbb U}}

\newcommand {\BZ}{{\mathbb Z}}
\newcommand {\BN}{{\mathbb N}}

\newcommand {\md}{/\!\!/}

\newcommand {\Ad}{{\mathrm{Ad\,}}}

\newcommand {\codim}{{\mathrm{codim\,}}}

\newcommand {\ind}{{\mathsf{ind\,}}}
\newcommand {\Lie}{{\mathsf{Lie\,}}}
\newcommand {\Ker}{\operatorname{Ker}}
\newcommand {\Ima}{{\mathrm{Im\,}}}
\newcommand {\Mor}{\operatorname{Mor}}

\newcommand {\rk}{{\mathrm{rk\,}}}
\newcommand {\rko}{{\mathrm{rk}}}
\newcommand {\spe}{{\mathsf{Spec\,}}}

\newcommand {\tr}{{\mathrm{tr\,}}}
\newcommand {\trdeg}{{\mathrm{tr.deg\,}}}

\newcommand {\tri}{\mathfrak{sl}_2}
\newcommand {\GR}[2]{{\textrm{{\bf #1}}}_{#2}}
\newcommand {\GRt}[2]{{\tilde{\textrm{{\bf #1}}}}_{#2}}
\newcommand {\ov}{\overline}
\newcommand {\un}{\underline}

\newcommand {\sfr}{\mathsf{R}}

\newcommand {\gig}{\mathsf{g.i.g.}}

\newcommand {\ctc}{\textsf{\bfseries C${\cdot}2{\cdot}$C}\ }
\newcommand {\ctrc}{\textsf{\bfseries C${\cdot}3{\cdot}$C}\ }
\newcommand {\cnc}{\textsf{\bfseries C${\cdot}n{\cdot}$C}\ }

\newcommand {\beq}{\begin{equation}}
\newcommand {\eeq}{\end{equation}}

\renewcommand{\le}{\leqslant}
\renewcommand{\ge}{\geqslant}

\newcommand {\bbk}{\Bbbk}

\begin{document}
\setlength{\parskip}{3pt plus 2pt minus 0pt}
\hfill { {\color{blue}\scriptsize May 6, 2017}}
\vskip1ex

\title{Semi-direct products of Lie algebras and covariants}
\author{Dmitri I. Panyushev}
\address[D.P.]%
{Institute for Information Transmission Problems of the R.A.S, Bolshoi Karetnyi per. 19, 
Moscow 127051, Russia}
\email{panyushev@iitp.ru}
\author[O.\,Yakimova]{Oksana S.~Yakimova}
\address[O.Y.]{Institut f\"ur Mathematik, Friedrich-Schiller-Universit\"at Jena,  07737 Jena, 
Deutschland}
\email{oksana.yakimova@uni-jena.de}
\thanks{The research of the first author was carried out at the IITP RAS at the expense of the Russian Foundation for Sciences (project {\rus N0} 14-50-00150).  The second author is partially supported by the DFG priority programme SPP 1388 
``Darstellungstheorie" and by Graduiertenkolleg GRK 1523 ``Quanten- und Gravitationsfelder".}
\begin{abstract}
The coadjoint representation of a connected algebraic group $Q$ with Lie algebra $\q$ is a thrilling and fascinating object. Symmetric invariants of $\q$ (= $\q$-invariants in the symmetric algebra $S(\q)$) can be considered as a first approximation to the understanding of the coadjoint action $(Q:\q^*)$ and coadjoint orbits. In this article, we study a class of non-reductive Lie algebras, where the description of the symmetric invariants is possible and the coadjoint representation has a number of nice invariant-theoretic properties. If $G$ is a semisimple group with Lie algebra $\g$ and $V$ is $G$-module, then we define $\q$ to be the semi-direct product of $\g$ and $V$. Then we are interested in the case, where the generic isotropy group for the $G$-action on $V$ is reductive and commutative. It turns out that in this case symmetric invariants of $\q$ can be constructed via certain $G$-equivariant maps from $\g$ to $V$ ("covariants").

\end{abstract}
\maketitle


\section*{Introduction}

\noindent
The coadjoint representation of an algebraic group $Q$ is a thrilling and fascinating object. It encodes 
information about many other representations of $Q$ and $\q=\Lie Q$. Yet, it is a very difficult object to 
study.  Symmetric invariants of $\q$ can be considered as a first approximation to the understanding of 
the coadjoint action $(Q:\q^*)$ 
and coadjoint orbits. The goal of this article is to describe and study a class of non-reductive Lie algebras, 
where the description of the symmetric invariants is possible and the 
coadjoint representation has a number of nice invariant-theoretic properties. The ground field $\bbk$ is algebraically closed and of characteristic $0$.

Let $G\to GL(V)$ be a (finite-dimensional rational) representation of a connected algebraic group $G$ with 
$\Lie G=\g$. We form a new Lie algebra $\q$ as the semi-direct product $\q=\g\ltimes V^*$, where 
$V^*$ is an abelian ideal. Then $Q=G\times V^*$ can be regarded as a connected algebraic group with 
$\Lie Q=\q$, where $1\ltimes V^*$ is a commutative unipotent normal subgroup. Here $\q^*=\g^*\oplus V$ 
and the algebra of symmetric invariants $\eus S(\q)^Q=\bbk[\q^*]^Q$ contains $\bbk[V]^G$ as a 
subalgebra. But finding the other invariants is a difficult and non-trivial problem. Nevertheless, 
one can use certain $G$-equivariant morphisms $F: V\to \g$ for constructing $Q$-invariants in 
$\bbk[\q^*]$. Our observation is that  if a generic stabiliser for 
$(G:V)$ is toral, then this is usually sufficient for obtaining a generating set for $\bbk[\q^*]^Q$.
\\ 
\indent
For $G$-modules $V$ and $N$, let $\Mor(V,N)$ denote the graded $\bbk[V]$-module of polynomial 
morphisms $F:V\to N$.  There is the natural map $\phi: \Mor(V,\g){\to} \Mor(V,V)$ such that  
$(\phi(F))(v):= F(v){\cdot}v$ for 
$v\in V$. If $F\in\Ker(\phi)$, then one obtains 
a $(1\ltimes V^*)$-invariant polynomial  $\hat F\in\bbk[\q^*]$ by letting 
$\hat F(\xi,v)=\langle F(v),\xi\rangle$ (Lemma~\ref{lm:V-invar}). Furthermore, if $F$ is also $G$-equivariant, then 
$\hat F\in \bbk[\q^*]^Q$.  Likewise, if $\Mor_G(V,N)$ denotes the $\bbk[V]^G$-module of 
$G$-equivariant morphisms ({\it covariants}), then there is the map 
$\Mor_G(V,\g)\stackrel{\phi_G}{\longrightarrow} \Mor_G(V,V)$, which is the 
restriction of $\phi$. Suppose that $G$ is reductive and $H\subset G$ is a generic isotropy group for
$(G:V)$, with $\h=\Lie H$. It is known that $\rk_{\bbk[V]}\Ker(\phi)=\dim \h$~\cite{jac}, and we prove that 
$\rk_{\bbk[V]^G}\Ker(\phi_G)=\dim \h^H$ whenever the action $(G:V)$ is stable (Theorem~\ref{thm:rangi}). 
Hence $\rk_{\bbk[V]}\Ker(\phi)=\rk_{\bbk[V]^G}\Ker(\phi_G)$ if
and only if the adjoint representation of $H$ is trivial; in particular, $\h$ must be toral. The main 
hope behind our considerations is that if $\Ker(\phi)$ is generated by $G$-equivariant morphisms, 
then $\bbk[V]^G$ and the polynomials $\hat F$ with $F\in \Ker(\phi_G)$ together generate the whole 
ring $\bbk[\q^*]^Q$. Actually, we prove this under certain additional constraints, see below.
For our general theorems, we also  
need the {\it codimension-2 condition} (=\,\textsf{\bfseries C${\cdot}2{\cdot}$C}) on the set  $V_{\sf reg}$ of $G$-regular elements in $V$. This 
means that $V\setminus V_{\sf reg}:=\{v\in V\mid \dim G{\cdot}v \text{ is not maximal}\}$ 
does not contain divisors.
\\ \indent
Our results concern the case in which $G$ is semisimple and \ctc holds for $(G:V)$.
Suppose that there are linearly independent homogeneous morphisms
$F_1,\dots,F_l\in \Ker(\phi)$ such that $l=\dim \h$ and  $\sum_i \deg F_i=\dim V-q(V\md G)$, where
$q(V\md G)$ is the minus degree of the Poincar\'e series of $\bbk[V]^G$.
Then we prove that $\Ker(\phi)$ is a free $\bbk[V]$-module with basis $F_1,\dots,F_l$ and
$\bbk[\q^*]^{1\ltimes V^*}\simeq \bbk[V][\hat F_1,\dots,\hat F_l]$ is a polynomial ring 
(Theorem~\ref{thm:main1}). Under certain additional assumptions (namely, $\h=\h^H$ and $H$ is 
not contained in a proper normal subgroup of $G$), we then prove that such $F_1,\dots,F_l$ are 
necessarily $G$-equivariant and hence $\Ker(\phi_G)$ is a free $\bbk[V]^G$-module and 
$\bbk[\q^*]^{Q}\simeq \bbk[V]^G[\hat F_1,\dots,\hat F_l]$. Furthermore, if $\bbk[V]^G$ is a polynomial ring, then the Kostant (regularity) criterion holds for $\q$ (Theorem~\ref{thm:main2}). In case $\dim\h=1$, our results are stronger and more precise, see Theorem~\ref{thm:l=1}.
\\ \indent
Using Elashvili's classification \cite{alela1,alela2}, one can write down the arbitrary representations 
of simple groups and irreducible representations of arbitrary semisimple groups with toral generic 
stabilisers. We then demonstrate that for most of these representations, the assumptions of our 
general theorems are satisfied. In each example, an emphasise is made on an explicit construction 
of morphisms $F_1,\dots,F_l$ and verification that they belong to $\Ker(\phi)$. In some cases, the 
construction is rather intricate and involved, cf. Examples~\ref{ex:tri-SL} and \ref{ex:sl-2-slag}.

The structure of the paper is as follows.
In Section~\ref{sect:prelim}, we gather some standard well-known facts on semi-direct products, 
regular elements, and generic stabilisers. In Section~\ref{sect:rank},  we consider the $\bbk[V]$-module 
of  polynomial morphisms $\Mor(V,\g)$ and the associated exact sequence
$0\to\Ker(\phi)\to \Mor(V,\g)\stackrel{\phi}{\to} \Mor(V,V)$.
We also compute the rank of the $\bbk[V]^G$-module  $\Ker(\phi_G)$.
Section~\ref{sect:semi} is the heart of the article. Here we present our main results on semi-direct 
products related to the case in which the \ctc holds for $(G:V)$, a generic stabiliser $\h$ for $(G:V)$ 
is toral, and there are linearly independent morphisms $F_1,\dots,F_l\in \Ker(\phi)$ such that 
$l=\dim \h$ and $\sum_{i=1}^l \deg F_i=\dim V-q(V\md G)$.
In Section~\ref{sect:codim2}, we explain how to verify that  the \ctc holds for a $G$-module $V$.
Examples of representations with toral generic stabilisers are presented in 
Sections~\ref{sect:primery} and \ref{sect:primery2}. For each example, we explicitly construct 
the morphisms $F_1,\dots,F_l$ such that the assumptions of our theorems from Section~\ref{sect:semi} 
are satisfied. Our results are summarised in Appendix~\ref{sect:tables}, where we provide tables of the representations with toral generic stabilisers.

This is a part of a general project initiated by the second author~\cite{Y}:  to classify all semi-direct products $\q=\g\ltimes V^*$ with semisimple $\g$ such that the ring  $\bbk[\q^*]^Q$ is polynomial.

{\sl \un{ Notation.}}
If an algebraic group $G$ acts on an irreducible affine variety $X$, then $\bbk[X]^G$ 
is the algebra of $G$-invariant regular functions on $X$ and $\bbk(X)^G$
is the field of $G$-invariant rational functions. If $\bbk[X]^G$
is finitely generated, then $X\md G:=\spe \bbk[X]^G$, and
the {\it quotient morphism\/} $\pi_{X,G}: X\to X\md G$ is induced by
the inclusion $\bbk[X]^G \hookrightarrow \bbk[X]$. If $X=V$ is a $G$-module, then $\N_G(V):=
\pi_{V,G}^{-1}(\pi_{V,G}(0))$ is the null-cone in $V$.
Whenever the ring $\bbk[X]^G$ is graded polynomial, 
the elements of any set of algebraically independent homogeneous generators 
will be referred to as {\it basic invariants\/}.
For a $G$-module  $V$ and $v\in V$, $\g_v=\{s\in\g\mid s{\cdot}v=0\}$ is the {\it stabiliser\/} of 
$v$ in $\g$ and $G_v=\{g\in G\mid g{\cdot}v=v\}$ is the {\it isotropy group\/} of $v$ in $G$.
\\ 
\textbullet \quad See also an explanation of the multiplicative (highest weight) notation for representations of semisimple groups in~\ref{notation-phi}.

\section{Preliminaries}
\label{sect:prelim}

\noindent
Let $G$ be a connected affine algebraic group with Lie algebra $\g$. The symmetric algebra 
$\gS (\g)$ is identified with the algebra of polynomial functions on $\g^*$ and we also write 
$\bbk[\g^*]$ for it.  
The algebra $\gS (\g)$ has the natural Poisson structure $\{\ ,\ \}$ such that
$\{x,y\}=[x,y]$ for $x,y\in\g$. A subalgebra $\gA\subset \gS (\g)$ is said to be {\it Poisson-commutative}, if it 
is a subalgebra in the usual (associative-commutative) sense and also $\{f,g\}=0$ for all $f,g\in\eus A$. The algebra of invariants $\gS (\g)^G=\bbk[\g^*]^G$ is the centraliser of $\g$ w.r.t $\{\ ,\ \}$, therefore it is the Poisson-centre of $\gS (\g)$.

\begin{opr}
 The index of $\g$, denoted $\ind\g$,  is $\min_{\xi\in\g^*}\dim\g_\xi$, where $\g_\xi$ is the stabiliser of $\xi$ with respect to the coadjoint representation of $\g$.
\end{opr}
\noindent
Set $b(\g)=(\dim\g+\ind \g)/2$. If $\g$ is reductive, then  
$\ind\g=\rk\g$ and $b(\g)$ equals the dimension of a Borel subalgebra. 
If 
$\eus A\subset \gS (\g)$ is Poisson-commutative, then 
\beq   \label{eq:ub}
\trdeg\eus A\le b(\g) .
\eeq

It is also known that this upper bound is always attained.

Let $V$ be a (finite-dimensional rational) $G$-module. 
The set of  $G$-{\it regular\/} elements of $V$ is defined to be
\[
   V_{\sf reg}=\{ v\in V\mid \dim G{\cdot}v \ge \dim G{\cdot}v' 
   \text{ for all } v' \in V\} \ .
\]
As is well-known, $V_{\sf reg}$ is a dense open subset of $V$ \cite{VP}. In particular, $\g^*_{\sf reg}$ is the set of $G$-regular elements w.r.t. the coadjoint representation of $G$.

\begin{opr}   \label{def:codim2}
We say that the {\it codimension-$n$ condition} (=\,\cnc) holds for the action $(G:V)$, if
$\codim_V (V\setminus V_{\sf reg})\ge n$.
\end{opr}%

\noindent
Suppose that $\trdeg \gS (\g)^G=\ind\g (=:l)$. Then
$\max_{\xi\in\g^*}\dim G\xi=\dim\g-l$. 
For any $f\in \gS(\g)$, let $(\textsl{d}f)_\xi\in\g$ denote the differential of $f$ at $\xi$.
We say that $\g$ satisfies {\it the Kostant (regularity) criterion\/} if the following properties hold for $\gS (\g)^G$ and
$\xi\in\g^*$:
\begin{itemize}
\item $\gS (\g)^G=\bbk[f_1,\dots,f_l]$ is a graded polynomial ring (with basic invariants $f_1,\dots,f_l$);
  \item $\xi\in\g^*_{\sf reg}$ if and only if $(\textsl{d}f_1)_\xi,\dots,(\textsl{d}f_l)_\xi$ are linearly
independent.
\end{itemize}
A very useful fact is that if \ctc holds for $(G:\g^*)$,  $\trdeg \gS (\g)^G=\ind\g=l$, and
there are algebraically independent $f_1,\dots,f_l\in \gS (\g)^G$ such that
$\sum_{i=1}^l \deg f_i=b(\g)$, then $f_1,\dots,f_l$ freely generate $\gS (\g)^G$ 
and the Kostant criterion holds for
$\g$, see~\cite[Theorem\,1.2]{coadj}.

{\it\bfseries Example.} If $\g$ is reductive and nonabelian, then 
$\codim (\g\setminus \g_{\sf reg})=3$. Hence the (co)adjoint representation of a 
reductive Lie algebra satisfies the \ctrc.

\noindent
For a $G$-module $V$, 
the vector space $\g\oplus V^*$ has a natural structure of Lie algebra, the {\it semi-direct product 
of\/ $\g$ and $V^*$}.
Explicitly, if $x,x'\in \g$ and $\zeta,\zeta'\in V^*$, then
\[
   [(x,\zeta), (x',\zeta')]=([x,x'], x{\cdot}\zeta'-x'{\cdot}\zeta) \ .
\]
This Lie algebra is denoted by $\q=\g\ltimes V^*$, and $V^*\simeq \{(0,\zeta)\mid \zeta\in V^*\}$ 
is an abelian ideal of $\q$. The corresponding connected algebraic group $Q$ is the semi-direct 
product of $G$ and the commutative unipotent group $\exp(V^*)\simeq V^*$. 
The group $Q$ can be identified with  $G\times V^*$,  the product being given by
\[
    (s,\zeta)(s',\zeta')= (ss', (s')^{-1}{\cdot}\zeta+\zeta'), \ \text{ where } \ s,s'\in G .
\]
In particular,  $(s,\zeta)^{-1}=(s^{-1}, -s{\cdot}\zeta)$. Then $\exp(V^*)$ can be identified with
$1\ltimes V^*:=\{(1,\zeta)\mid \zeta\in V^*\} \subset G\ltimes V^*$.
If $G$ is reductive, then the subgroup $1\ltimes V^*$ is the unipotent radical of $Q$, also denoted
by $R_u(Q)$.

Let $\mu: V\times V^* \to \g^*$ be the {\it moment map}, i.e., 
$\mu(v,\zeta)(g):=\langle \zeta, g{\cdot}v\rangle$, where $g\in\g$ and $\langle\ ,\ \rangle$ is the pairing  of $V$ and $V^*$. The restriction of the coadjoint representation of $Q$ to $1\ltimes V^*$ is explicitly described as follows. If $\zeta\in V^*$ and $\eta=(\xi,v)\in \q^*=\g^*\times V$, then
\beq    \label{eq:coadj-V}
   (1\ltimes \zeta){\cdot}\eta=(\xi+\mu(v,\zeta), v) .
\eeq
Since $\mu(v,\zeta)=0$ if and only if $\zeta\in (\g{\cdot}v)^\perp$, the maximal dimension of the
$(1\ltimes V^*)$-orbits in $\q^*$ equals $\max_{v\in V}\dim (\g{\cdot}v)=\dim\g-\min_{v\in V}\dim \g_v$.

\begin{lm}   \label{lm:vspomogat}
 For $\q=\g\ltimes V^*$. There is a dense open subset $\tilde\Omega\in V_{\sf reg}$ such that for 
 any $x\in\tilde \Omega$
 \begin{itemize}
\item[\sf (i)] \   $b(\q)=\dim V +b(\g_x)$;
\item[\sf (ii)] \  $\trdeg (\bbk[\q^*]^{1\ltimes V^*})=\dim V+\dim\g_x$.
\end{itemize}
\end{lm}
\begin{proof}
(i) \ By \cite{rais}, there is a dense open subset  $\tilde\Omega\in V_{\sf reg}$ such that
$\ind\q=\dim V-\max_{v\in V} \dim\g{\cdot}v +\ind \g_x=
\dim V-\dim\g+\dim\g_x+\ind\g_x$ for any $x\in\tilde\Omega$. This yields the desired formula for $b(\q)$.

(ii) By Rosenlicht's theorem~\cite[2.3]{VP}, 
\[
\trdeg (\bbk[\q^*]^{1\ltimes V^*})=\dim\q-\max_{\eta\in\q^*}
\dim\bigl( (1\ltimes V^*){\cdot}\eta\bigr)=\dim\q-\dim\g+\dim\g_x . \qedhere
\]
\end{proof}

It follows from this lemma that $\trdeg (\bbk[\q^*]^{1\ltimes V^*})\ge b(\q)$ and the equality holds 
if and only if $\ind\g_x=\dim\g_x$, i.e., $\g_x$ is abelian for generic elements of $V$. 
By \cite{Y16}, if there is a dense open subset $\tilde\Omega$ of $V$ such that $\g_x$ is abelian for 
all $x\in \tilde\Omega$, then $\eus P:=\bbk[\q^*]^{1\ltimes V^*}$ is Poisson-commutative.
Having in mind the general upper bound \eqref{eq:ub}, we conclude that in such a case  $\eus P$ is 
a Poisson-commutative subalgebra of $\bbk[\q^*]$ of maximal dimension.  Moreover, since $\eus P$ is the centraliser of $V^*$ in $(\gS (\q),\ \{\ ,\ \})$, it is also a maximal Poisson-commutative subalgebra, cf.~\cite[Theorem\,3.3]{BSM14}.

We say that the action $(G{:}V)$ {has a generic stabiliser}, if there exists
a dense open subset $\Omega\subset V$ such that all stabilisers $\g_v$, $v\in \Omega$, are 
$G$-conjugate. Then any subalgebra $\g_v$, $v\in\Omega$, is called a {\it generic stabiliser}
(=\,\textsf{g.s.}).
Likewise, one defines  a {\it generic isotropy group} (=\,$\gig$), 
which is a subgroup of $G$. By~\cite[\S\,4]{Ri}, the linear action $(G:V)$ has a generic stabiliser if and only if it has a generic isotropy group. It is also known that $\gig$ always exists if $G$ is reductive.
A systematic treatment of generic stabilisers in the context of reductive group
actions can be found in \cite[\S 7]{VP}.

\section{On the rank of certain modules of covariants}
\label{sect:rank}

\noindent
For finite-dimensional $\bbk$-vector spaces $V$ and $N$, let $\Mor(V,N)$ denote the set of 
polynomial morphisms $F: V\to N$. Clearly, $\Mor(V,N)\simeq \bbk[V]\otimes N$ and it is a free graded 
$\bbk[V]$-module of rank $\dim N$. Here $\deg F=d$, if $F(tv)=t^dF(v)$ for any $t\in \bbk^\times$ and 
$v\in V$.

If both $V$ and $N$ are $G$-modules, then $G$ acts on $\Mor(V,\g)$ by $(g\ast F)(v)=g(F(g^{-1}v))$. 
Therefore, $g\ast F=F$ for all $g\in G$ if and only 
if $F$ is $G$-equivariant. Write $\Mor_G(V,N)$ for the set
of $G$-{\it equivariant} polynomial morphisms $V\to N$. It is also called the 
{\it module of covariants of type\/} $N$. We have 
$\Mor_G(V,N)\simeq (\bbk[V]\otimes N)^G$. In the rest of the section, we assume that $G$ is reductive. 
Then $\Mor_G(V,N)$ is a finitely generated $\bbk[V]^G$-module, see e.g. \cite[3.12]{VP}.

Given a $G$-module $V$, consider the exact sequence of $\bbk[V]$-modules
\[
   0\to\Ker(\phi)\to \Mor(V,\g)\stackrel{\phi}{\to} \Mor(V,V)\ ,
\]
where $\phi(F)(v):= F(v){\cdot}v$ for $F\in \Mor(V,\g)$ and $v\in V$.
Therefore, 
\[
    \Ker(\phi)=\{ F\in \Mor(V,\g)\mid F(v)\in \g_v \quad\forall v\in V\} \ .
\]
Here $\rk\phi=\max_{v\in V}\dim \g{\cdot}v$ \cite[Prop.\,1.7]{jac} and hence
$\rk \Ker(\phi)=\min_{v\in V}\dim\g_v$. Recall that if $R$ is a domain and $M$ is a finitely generated 
$R$-module, then the rank of $M$ is $\rk M=\rko_R(M)=\dim_{\mathsf{Quot}(R)}M\otimes 
{\mathsf {Quot}(R)}$. 
\\ \indent
We also consider the ``equivariant sequence'' that comprises $\bbk[V]^G$-modules:
\[
   0\to\Ker(\phi_G) \to \Mor_G(V,\g)\stackrel{\phi_G}{\longrightarrow} \Mor_G(V,V) \ .
\]
Here $\phi_G$ is the restriction of $\phi$ to $\Mor_G(V,\g)$. We are interested in conditions under which the $\bbk[V]$-module $\Ker(\phi)$ is generated by $G$-equivariant morphisms. In other words, when is it true that $\Ker(\phi)\simeq \bbk[V]\otimes_{\bbk[V]^G} \Ker(\phi_G)$ ?

If $H$ is a generic isotropy group  for $(G:V)$ and $\h=\Lie H$, then
we write $\h=\mathsf{g.s.}(\g:V)$ and $H=\gig(G:V)$ for this. Then $\min_{v\in V}\dim \g_v=\dim\h$ and hence
\beq    \label{eq:rk-ker-hat-phi}
    \rk \Ker(\phi)=\dim\h .
\eeq
Recall that the $G$-action on $V$ is said to be {\it stable}, if
the union of closed $G$-orbits is dense in $V$, see \cite[\S\,7]{VP}. Then
 $H$ is a reductive (not necessarily connected) group.
By a general result of Vust \cite[Chap.\,III]{vust}, if the action $(G:V)$ is stable, then
\begin{equation}         \label{vust-rang}
\text{the rank of the $\bbk[V]^G$-module $\Mor_G(V,N)$ equals $\dim N^H$}.
\end{equation}

For the reader's convenience, we outline a proof:
\\ \indent
\textbullet \quad If $F$ is $G$-equivariant, then $F(v)\in N^{G_v}$ for any $v\in V$. Applying this to the open set of $G$-generic elements in $V$, we obtain that $\rk \Mor_G(V,N)\le \dim N^H$.
\\ \indent
\textbullet \quad On the other hand, the "evaluation" map 
$\epsilon_v: \Mor_G(V,N) \to N^{G_v}$, $F\mapsto F(v)$, is onto whenever $G{\cdot}v=\ov{G{\cdot}v}$, see \cite[Theorem\,1]{indag02}. Hence if generic $G$-orbits in $V$ are closed (and isomorphic to $G/H$), then the upper bound $\dim N^H$ is attained.

Our goal is to compute the rank of the $\bbk[V]^G$-module $\Ker(\phi_G)$.

\begin{thm}      \label{thm:rangi}  
If the action $(G{:}V)$ is stable and $H=\gig(G{:}V)$, then $\rk\Ker(\phi_G)=\dim \h^H$.
\end{thm}
\begin{proof}
The reductive group $W=N_G(H)/H$ acts on $V^H$.
By the Luna-Richardson theorem \cite{lr79}, the restriction homomorphism
$\bbk[V]\to \bbk[V^H]$ induces an isomorphism of rings of invariants
$\bbk[V]^G\simeq\bbk[V^H]^W$. This common ring will be denoted by $\mathfrak J$.
Consider the commutative diagram of $\mathfrak J$-modules
\[
\begin{array}{ccccccc}
 0 &\to&\Ker(\phi_G) &\to &\Mor_G(V,\g)&\stackrel{\phi_G}{\longrightarrow}& 
 \Mor_G(V,V)\\
    &    &     &        & \downarrow &    & \downarrow   \\
0 &\to&\Ker(\psi_W) &\to &\Mor_W(V^H,\g^H)&\stackrel{\psi_W}{\longrightarrow}& 
\Mor_W(V^H,V^H),
\end{array}
\]
where the vertical arrows denote the restriction
of  $G$-equivariant morphisms to $V^H\subset V$.
Note that the $W$-module $\g^H$ is not the Lie algebra of $W$.
However, the $\mathfrak J$-module homomorphism $\psi_W$ is being defined
similarly to $\phi_G$.
By construction, the action $(W{:}V^H)$ is again stable and 
has trivial generic isotropy groups. Therefore, using Eq.\,\eqref{vust-rang},
 we conclude that
\begin{gather*}
\rk\Mor_W(V^H,\g^H)=\dim \g^H=\rk\Mor_G(V,\g)   \\
\rk\Mor_W(V^H,V^H)=\dim V^H= \rk\Mor_G(V,V) .
\end{gather*}
Since $H$ is a generic isotropy group, $\ov{G{\cdot}V^H}=V$.
It follows that both vertical arrows are {\it injective\/} homomorphisms of $\mathfrak J$-modules
of equal ranks. Therefore, they give rise to isomorphisms over the field of fractions of $\mathfrak J$ and hence
$\rk \Ker(\psi_W)=\rk \Ker(\phi_G)$. Here
\begin{multline*}
   \Ker(\psi_W)=\{F: V^H\to \g^H \mid  F\ \text{ is $W$-equivariant and }\ 
   F(v)\in (\g^H)_v\ \ \forall v\in V^H\}=\\
 =\{ F \mid F\ \text{ is $W$-equivariant and }\ F(v)\in \h^H\ \ \forall v\in V^H\}
 \simeq \Mor_W(V^H,\h^H).
\end{multline*}
The second equality follows from the fact that $\g_v=\h$ for a generic
$v\in V^H$ and hence $F(v)\in \h$ for any $v\in V^H$.
Since  $\gig(W{:}V^H)=\{1\}$, 
Eq.\,\eqref{vust-rang} implies that 
$\rk \Ker(\psi_W)=\dim\h^H$.
\end{proof}

Comparing Eq.~\eqref{eq:rk-ker-hat-phi} and Theorem~\ref{thm:rangi} provides the following 
necessary condition:

\begin{cl}
If the action $(G:V)$ is stable and the $\bbk[V]$-module $\Ker(\phi)$ is generated by 
$G$-equivariant morphisms, then $\h=\h^H$ (i.e., the adjoint representation of $H$ is trivial).
In particular, $\h$ is a toral subalgebra of $\g$.
\end{cl}

There are several cases in which this condition on $\h$ is also sufficient.

$\bullet$ \ If $(G:V)$ is the isotropy representation of a symmetric variety, then the condition that
$\h$ is toral does imply that $\Ker(\phi)$ is a free\/ $\bbk[V]$-module generated 
by $G$-equivariant morphisms, see \cite[Theorem\,5.8]{coadj}. 

$\bullet$ \ If $H$ is finite, then $\Ker(\phi)$ is a trivial $\bbk[V]$-module.

Next, we provide one more good case. For  $F\in \Mor(V,N)$, let ${\eus V}(F)$ denote the 
set of zeros of $F$. If $\dim N=1$, then $F$ is a polynomial function on $V$ and
${\eus V}(F)$ is a divisor.

\begin{prop}      \label{dim=1}
Suppose that $G$ is semisimple and $\gig(G:V)$ is a one-dimensional torus. 
Then $\Ker(\phi)$ is a free\/ $\bbk[V]$-module of rank 1  generated by 
a $G$-equivariant morphism.
\end{prop}
\begin{proof}
Since $G$ is semisimple and $\gig$ is reductive, the action $(G:V)$ is stable~~\cite[Theorem\,7.15]{VP}. Hence  
$\rk \Ker(\phi_G)=1$ in view of Theorem~\ref{thm:rangi}. Then we can pick a nonzero homogeneous primitive element 
$F\in \Ker(\phi_G)$, i.e., $F$ cannot 
be written as $f\check F$, where $\check F\in \Ker(\phi_G)$ and $f\in \bbk[V]^G$ with $\deg f>0$. 
Then $F$ is also primitive as element of $\Mor(V,\g)$.
Indeed, assume that $F=f\check F$, where $\check F\in \Mor(V,\g)$, $f\in \bbk[V]$ and
$\deg f>0$. 
Because $F$ is a $G$-equivariant morphism,
${\eus V}(F)$ is $G$-stable. Since  ${\eus V}(f)\subset{\eus V}(F)$ and ${\eus V}(f)$ is a divisor, 
${\eus V}(f)$ is necessarily a $G$-stable divisor
in $V$. Because $G$ is semisimple, $f\in \bbk[V]^G$.
It follows that $\check F\in \Mor_G(V,\g)$. The relation $F=f\check F$ shows that 
$\check F(v)\in \g_v$ for any $v\in V\setminus {\eus V}(f)$.
Hence $\check F(v)\in\g_v$ for any $v\in V$, and this contradicts the primitivity of $F$ in $\Ker(\phi_G)$.

Let $\tilde F\in \Ker(\phi)$\/ be an arbitrary homogeneous element.
Since $\rk\Ker(\phi)=1$, there are coprime homogenous 
$f,\tilde f\in \bbk[V]$ such that 
$fF=\tilde f\tilde F$. If $\deg\tilde f>0$, then ${\eus V}(\tilde f)\subset{\eus V}(F)$ and, as 
in the previous paragraph, this leads to a contradiction.
Thus, $\tilde f$ is invertible, and we are done.
\end{proof}%
Using the theory to be developed in Section~\ref{sect:semi}, we provide a number of non-trivial examples
of representations with toral generic stabilisers such that $\Ker(\phi)$ is generated by $G$-equivariant morphisms, see Sections~\ref{sect:primery} and \ref{sect:primery2}.

\section{Semi-direct products  with good invariant-theoretic properties} 
\label{sect:semi}

\noindent
In this section, we describe a class of representations $(G:V)$ such that $\Ker(\phi)$ is generated by 
$G$-equivariant morphisms, $\q=\g\ltimes V^*$ satisfies the 
Kostant criterion, and $(Q:\q^*)$ has nice invariant-theoretic properties.
\\ \indent
For $F\in\Mor(V,\g)$ and  $\eta=(\xi,v)\in\q^*=\g^*\times V$, we define $\hat F\in\bbk[\q^*]$ by 
$\hat F(\eta):=\langle F(v),\xi\rangle$, where $\langle\ ,\ \rangle$ denote the pairing of dual 
spaces.

\begin{lm}     \label{lm:V-invar}
We have $\hat F\in \bbk[\q^*]^{1\ltimes V^*}$ if and only if $F(v){\cdot}v=0$ for all $v\in V$, i.e., 
$F\in\Ker(\phi)$.
\end{lm}
\begin{proof}
By \eqref{eq:coadj-V}, the invariance with respect to $1\ltimes V^*$ means that
\[
\langle F(v),\xi\rangle=\langle F(v),\xi+\mu(v,\zeta)\rangle,   \quad (\xi,v)\in\q^* ,
\]
for any $\zeta\in V^*$. Hence $0=\langle F(v),\mu(v,\zeta)\rangle=\langle F(v){\cdot}v,\zeta\rangle$, 
and we are done.
\end{proof}

Thus, any $F\in\Ker(\phi)$ gives rise to $\hat F\in \bbk[\q^*]^{1\ltimes V^*}$. Moreover, it is clear that
if $F$ is $G$-equivariant, then $\hat F\in \bbk[\q^*]^Q$. It follows from Eq.~\eqref{eq:coadj-V} that
if $\zeta\in V^*$ is regarded as a linear function on $\q^*=\g^*\times V$, then
$\zeta$ is $1\ltimes V^*$-invariant.  Hence

\textbullet \quad both $\eus S(V^*)=\bbk[V]$ and $\{\hat F\mid F\in \Ker(\phi)\}$ belong to 
$\bbk[\q^*]^{1\ltimes V^*}$;

\textbullet \quad both $\bbk[V]^G$ and $\{\hat F\mid F\in \Ker(\phi_G)\}$ belong to 
$\bbk[\q^*]^{Q}$;
\\
We provide below certain conditions that guarantee us that $\bbk[\q^*]^{1\ltimes V^*}$ and
$\bbk[\q^*]^{Q}$ are generated by the respective subsets.

Recall some properties to the symmetric invariants of semi-direct products: 
\begin{itemize}
\item[\sf (i)] \ The decomposition $\q^*=\g^*\oplus V$ yields a bi-grading of 
$\bbk[\q^*]^Q$~\cite[Theorem\,2.3(i)]{coadj}. The same argument proves that the algebra 
$\bbk[\q^*]^{1\ltimes V^*}$ is also bi-graded.
\item[\sf (ii)] \ The algebra $\bbk[V]^G$ is contained in $\bbk[\q^*]^Q$. Moreover, a minimal generating
system for $\bbk[V]^G$ is a part of a minimal generating system of 
$\bbk[\q^*]^S$~\cite[Sect.\,2\,(A)]{coadj}. In particular, if $\bbk[\q^*]^Q$ is a polynomial ring, then so is 
$\bbk[V]^G$.
\end{itemize}
\begin{rmk}    \label{rem:iz-kota}
Note that $\hat F$ associated with $F\in\Ker(\phi)$ has degree $1$ w.r.t. $\g$. Conversely, it can be 
shown that if $f\in \bbk[\q^*]^{1\ltimes V^*}$ has degree 1 w.r.t. $\g$, then $f=\hat F$ for some 
$F\in\Ker(\phi)$, see \cite[Lemma\,2.1]{Y16}.
In other words, there is a natural bijection $\Ker(\phi) \stackrel{1:1}{\longleftrightarrow} (\g\otimes \bbk[V])^{1\ltimes V^*}$.
It is also true that $\Ker(\phi_G) \stackrel{1:1}{\longleftrightarrow }(\g\otimes \bbk[V])^{G\ltimes V^*}$.
\end{rmk}

If $G\subset GL(V)$ is reductive, then $\bbk[V]^G$ is finitely generated and $q(V\md G)$ 
stands for the minus degree of the Poincar\'e series of the graded algebra $\bbk[V]^G$. More precisely, $\bbk[V]^G=\bigoplus_{j\in \BN} \bbk[V]^G_j$ and its Poincar\'e series
is  
\[
  \eus F(\bbk[V]^G; t)=\sum_{j\in \BN} \dim \bbk[V]^G_j t^j .
\]
Here $\eus F(\bbk[V]^G; t)=P(t)/\tilde P(t)$ is a rational function and, by definition, 
$q(V\md G)=\deg \tilde P-\deg P$. In particular, if $\bbk[V]^G$ is a polynomial ring, then $q(V\md G)$ equals the sum of degrees of the basic invariants. By~\cite[Korollar\,5]{kn86}, if $G$ is semisimple, then
$q(V\md G)\le q(V)=\dim V$. The arbitrary representations of simple algebraic groups and the irreducible representations of semisimple groups such that $q(V\md G)< \dim V$ are classified in~\cite{kl87}.

\noindent
Recall some properties of the linear actions of semisimple groups. If $G\subset GL(V)$ is semisimple, then
\begin{itemize}
\item $\bbk(V)^G$ is the quotient field of $\bbk[V]^G$, hence
    $\max_{v\in V} \dim G{\cdot}v=\dim V-\dim V\md G$~\cite{VP};
\item $(G:V)$ is stable if and only if \ $\gig(G:V)$ is reductive~\cite[Theorem\,7.15]{VP}.
\end{itemize}

\begin{thm}   \label{thm:main1}
Let $G\subset GL(V)$ be semisimple and  $l=\min_{v\in V}\dim\g_v=\dim \gig(G:V)>0$.
Suppose that $\codim(V\setminus V_{\sf reg})\ge 2$ and there are linearly independent 
(over $\bbk[V]$) homogeneous morphisms $F_1,\dots,F_l\in\Ker(\phi)$ such that
\beq   \label{eq:summa}
   \textstyle \sum_{i=1}^l \deg F_i  + q(V\md G)=\dim V .
\eeq
Then 
\begin{itemize}
\item[\sf (i)] \  $F_1(v),\dots, F_l(v)\in \g$ are linearly independent for all $v\in V_{\sf reg}$ 
and $\bigwedge_{i=1}^l F_i: V\to \wedge^m\g$ is $G$-equivariant;
\item[\sf (ii)] \  $\Ker(\phi)$ is a free $\bbk[V]$-module of rank $l$, with basis $F_1,\dots,F_l$;
\item[\sf (iii)] \ $\bbk[\q^*]^{R_u(Q)}=\bbk[V][\hat F_1,\dots,\hat F_l]$, that is,
$\q^*\md {R_u(Q)}\simeq V\times \mathbb A^l$;
\item[\sf (iv)] \ The $\bbk$-linear span of $F_1,\dots,F_l$ (resp. $\hat F_1,\dots,\hat F_l$)  is a $G$-stable subspace of\/ $\Mor(V,\g)$ (resp. $\bbk[\q^*]$).
\end{itemize}
\end{thm}
\begin{proof}
(i) \ Since a generic isotropy group is $l$-dimensional,
$\max_{v\in V}\dim G{\cdot}v =\dim\g-l=:m$. By~\cite[Satz~1 \& Korollar~4]{kn86}, there is a 
$G$-equivariant map $c:V\to \wedge^m\g^*\simeq \wedge^l\g$ such that $\deg c=\dim V- q(V\md G)$ 
and if $v\in V_{\sf reg}$, then $0\ne c(v) \in \wedge^l(\g_v)\subset \wedge^l\g$. On the other hand,
the map $\tilde c=\bigwedge_{i=1}^l F_i: V\to \wedge^l\g$ has the same degree and also
$\tilde c(v) \in \wedge^l(\g_v)\subset \wedge^l\g$ for almost all $v\in V_{\sf reg}$. In other words,
$c(v)$ and $\tilde c(v)$ are proportional for almost all $v\in V$. Consequently, there are coprime 
homogeneous $f,\tilde f\in \bbk[V]$ such that  $fc=\tilde f\tilde c$. Since $\deg c=\deg\tilde c$, 
we have
$\deg f=\deg\tilde f$ as well. If $\deg \tilde f>0$, then there is $v\in V_{\sf reg}$ such that $\tilde f(v)=0$ and $f(v)\ne 0$. Then $c(v)=0$, a contradiction! Hence $f,\tilde f\in \bbk^\times$,
$\bigwedge_{i=1}^l F_i: V\to \wedge^m\g$ is $G$-equivariant,  and
$F_1(v),\dots, F_l(v)\in \g$ are linearly independent for all $v\in V_{\sf reg}$.
\\ \indent
(ii) \ As $\codim(V\setminus V_{\sf reg})\ge 2$, the last property also implies that $(F_1,\dots, F_l)$ is 
a basis for the $\bbk[V]$-module $\Ker(\phi)$.
Indeed, recall that $\rk \Ker(\phi)=\min_{v\in V}\dim\g_v=l$.
If $F\in \Ker(\phi)$, then there are $f,f_i\in \bbk[V]$ such that
$fF=\sum_{i=1}^l f_iF_i$. Again, if $\deg f>0$, then there is $v\in V_{\sf reg}$ such that $f(v)=0$ and
$f_i(v)\ne0$ for all (some) $i$. This contradicts the linear independence of  $\{F_i(v)\}$ for all 
$v\in V_{\sf reg}$. Hence $f\in \bbk^\times$, and we are done.

(iii) \ Recall that now $R_u(Q)=1\ltimes V^*$, $\mu: V\times V^*\to \g^*$ is the moment mapping, and
the $R_u(Q)$-orbits in $\q^*$ are 
\[
    R_u(Q){\cdot}(\xi,v)=(\xi+\mu(v, V^*), v)=(\xi+(\g_v)^\perp, v) .
\]
Hence $\dim R_u(Q){\cdot}(\xi,v)=\dim (\g_v)^\perp$ and 
$\max_{\eta\in\q^*} \dim R_u(Q){\cdot}\eta= \dim\g-l$.  Therefore 
$\trdeg \bbk[\q^*]^{R_u(Q)}=\dim V+l$. Let  $(\zeta_1,\dots,\zeta_n)$, $n=\dim V$, be a basis of 
$V^*$ (We regard the $\zeta_i$'s as linear functions on $\q^*$.) 
Then $\zeta_1,\dots,\zeta_n, \hat F_1,\dots, \hat F_l$ are algebraically independent and belong to 
$\bbk[\q^*]^{R_u(Q)}$. Consider the map
$\pi:\q^*=\g^*\oplus V \to V\times \mathbb A^l$ given by
\beq   \label{eq:pi}
   \bigl(\eta=(\xi, v) \in \q^* \bigr) \mapsto   \bigl((v, \hat F_1(\eta),\dots,\hat F_l(\eta)\in
   V\times \mathbb A^l\bigr). 
\eeq
By the Igusa Lemma~\cite[Theorem\,4.12]{VP}, in order to prove that $\pi$ is the quotient morphism by
$R_u(Q)$,
it suffices to verify the following two conditions:

$(\lozenge_1)$ \ The closure of $(V\times\mathbb A^l)\setminus \Ima(\pi)$ does not contain divisors;

$(\lozenge_2)$ \ There is a dense open subset $\Psi\subset V\times\mathbb A^l$ such that $\pi^{-1}(b)$ contains a dense $R_u(Q)$-orbit for all $b\in\Psi$.

\noindent
{\it\bfseries For} $(\lozenge_1)$:  If $v\in V_{\sf reg}$, then $\{F_i(v)\}$ are linearly independent in view of (i). Therefore, the system of linear equations $\langle F_i(v),\xi\rangle=a_i$, $1\le i\le l$,  has a solution $\xi$ for any 
$(a_1,\dots,a_l)\in\mathbb A^l$. Therefore, $\Ima(\pi)\supset V_{\sf reg}\times \mathbb A^l$.

\noindent
{\it\bfseries For} $(\lozenge_2)$: Suppose that $v\in V_{\sf reg}$, $\bar a=(a_1,\dots,a_l)\in \mathbb A^l$, and 
$\xi_0$ is a solution to the system $\langle F_i(v),\xi\rangle=a_i$. Then $\pi^{-1}(v,\bar a)=
(\xi_0+ (\g_v)^\perp, v)$, which is a sole $R_u(Q)$-orbit.

Thus,  $\bbk[\q^*]^{R_u(Q)}=\bbk[\zeta_1,\dots,\zeta_n, \hat F_1,\dots, \hat F_l]$ 
and the morphism  $\pi_{\q^*,R_u(Q)}$ is given by \eqref{eq:pi}.

(iv) \ Since $\q^*\md {R_u(Q)}\simeq V\times \mathbb A^l$, $G=Q/R_u(Q)$ acts on $\q^*\md {R_u(Q)}$, and $V$ is a $G$-module, the explicit form of the free generators of 
$\bbk[\q^*]^{R_u(Q)}$ shows that the $\bbk$-linear span $\langle \hat F_1,\dots, \hat F_l\rangle$ is a $G$-stable subspace of $\bbk[\q^*]$. Using the definition of $\hat F_i$, one readily verifies that 
$g{\cdot}\hat F_i=\widehat{g{\ast} F_i}$. This means that 
\[
  \text{$\langle \hat F_1,\dots, \hat F_l\rangle$ is a $G$-stable subspace}  \Leftrightarrow
  \text{$\langle F_1,\dots, F_l\rangle$ is a $G$-stable subspace} .  \qedhere
\]
\end{proof}

Note that part (ii) of this theorem is a direct consequence of (i),  and our proof of (ii), i.e., essentially the proof of the implication (i)$\Rightarrow$(ii), appears already in the proof of Theorem~1.9 in~\cite{jac}.

The condition \eqref{eq:summa} is rather strong, and all known to us instances of such a 
phenomenon occur only if $\mathsf{g.s.}(\g:V)$ is abelian, see examples in Sections~\ref{sect:primery} and
\ref{sect:primery2}. As a by-product of our proof of part (i) in Theorem~\ref{thm:main1}, we also obtain the
following assertion:

\begin{lm}   \label{lm:neravenstvo-for-summa} 
Suppose that $G\subset GL(V)$ is semisimple, $\codim(V\setminus V_{\sf reg})\ge 2$, $\dim \gig(G:V)=l$, and 
$F_1,\dots,F_l\in \Ker(\phi)$ are homogeneous and  linearly independent. Then
$\sum_{i=1}^l \deg F_i\ge \deg c= \dim V-q(V\md G)$.
\end{lm}

\begin{rmk}
The interest of Theorem~\ref{thm:main1} is in the case, where $l=\dim \mathsf{g.s.}(\g:V)> 0$, i.e., there {\bf are} certain morphisms $\{F_i\}$. If $l=0$, then the {\sl codimension-$2$ condition} for $(G:V)$ implies that $q(V\md G)=q(V)$ \cite[Korollar\,4]{kn86}. i.e., formally, Eq.~\eqref{eq:summa} holds. Then parts (i), (ii), (iv) become vacuous, but part (iii) still makes sense and remains true. 
For, in this case $\bbk[\q^*]^{R_u(Q)}\simeq \bbk[V]$, see~\cite[Theorem\,6.4]{p05}.
\end{rmk}

\begin{thm}   \label{thm:main2}
Let $G, V, F_1,\dots, F_l$ be as in Theorem~\ref{thm:main1}. Suppose also that the identity
component of $H=\gig(G:V)$ is a torus,  $H$ is not contained in a proper normal 
subgroup of $G$, and $\h^H=\h$. Then 
\begin{itemize}
\item[\sf (i)] \ If the \cnc holds for $(G:V)$ with $n\ge 2$, then it also holds for $(Q:\q^*)$;
\item[\sf (ii)] \  The morphisms $F_1,\dots, F_l$ are $G$-equivariant, 
the corresponding polynomials $\hat F_1,\dots, \hat F_l$ are $G$-invariant,
and hence $\bbk[\q^*]^Q=\bbk[V]^G[\hat F_1,\dots,\hat F_l]$, i.e.,
$\q^*\md Q \simeq V\md G \times \mathbb A^l$;
\item[\sf (iii)] \  $\bbk[\q^*]^{R_u(Q)}$ is a maximal Poisson-commutative subalgebra of $\bbk[\q^*]$;
\item[\sf (iv)] \  If \ $\bbk[V]^G$ is a polynomial algebra, then the Kostant criterion holds for $\q$.
\end{itemize}
\end{thm}
\begin{proof}
(i) \ Since a generic stabiliser is abelian, the standard deformation argument shows that $\g_v$ is 
abelian for any $v\in V_{\sf reg}$. It then follows from \cite[Prop.\,5.5]{p05} that $(\xi,v)$
is $Q$-regular for any $\xi\in\g^*$. Hence  $\q^*_{\sf reg}\supset \g^*\times V_{\sf reg}$.

(ii) \ By Theorem~\ref{thm:main1}(iv), the space $\langle F_1,\dots, F_l\rangle$ is $G$-stable 
and therefore
\[
    g\ast F_i =\sum_{j=1}^l a_{ij}(g) F_j   \qquad \forall g\in G .
\]
Recall that $g(F_i(g^{-1}v))=(g\ast F_i)(v)$. 
If  $g\in G_v$, then $F_i(g^{-1}v)=F_i(v)\in \g_v$. Moreover, if $G_v\sim H$, then
$g(F_i(g^{-1}v))=F_i(v)$ in view of the assumption $\h^H=\h$. Therefore
$F_i(v)=\sum_{j=1}^l a_{ij}(g) F_j$ for all $v$ such that $G_v\sim H$ and $g\in G_v$.
By Theorem~\ref{thm:main1}(i),  $\{F_i(v)\}$ are linearly independent. Hence
$a_{ij}(g)=\delta_{ij}$ for any $g\in G_v$ and $G_v\sim H$. Hence the kernel of the representation 
$\rho: G\to GL(\langle F_1,\dots, F_l\rangle)$ contains the normal subgroup generated by all
generic isotropy subgroups. Under our assumption, this implies $\Ker(\rho)=G$. Therefore, each
$F_i$ is $G$-equivariant and thereby each $\hat F_i$ is $G$-invariant and also $Q$-invariant.
Hence $G$ acts trivially on $\mathbb A^l$ and 
\[
   \q^*\md Q =(\q^*\md {R_u(Q)})\md G \simeq (V\times \mathbb A^l)\md G\simeq
   V\md G \times \mathbb A^l .
\]
\indent (iii) \ This is a particular case of more general results of \cite{Y16}. However, using the $G$-equivariance of  $\{F_i\}$ one can verify directly that the basic invariants in $\bbk[\q^*]^{R_u(Q)}$ pairwise commute w.r.t. the Poisson bracket $\{\ ,\ \}$
(cf. the proof of Theorem~3.3 in \cite{BSM14}).

(iv) \ If  $\bbk[V]^G$ is a polynomial algebra, then so is $\bbk[\q^*]^Q$ (in view of (ii)), and the sum of degrees of 
the basic invariants in $\bbk[\q^*]^Q$ equals 
$q(V\md G)+\sum_{i=1}^l \deg\hat F_i=q(V\md G)+(\sum_{i=1}^l \deg F_i)+l=\dim V+l=b(\q)$. 
Together with the \ctc for $(Q:\q^*)$, this implies that the Kostant criterion holds for $\q$,
see~\cite[Theorem\,1.2]{coadj}. 
\end{proof}

\begin{cl}  Under the assumptions of Theorems~\ref{thm:main1} and \ref{thm:main2},
the $\bbk[V]$-module $\Ker(\phi)$ is free and is generated by $G$-equivariant morphisms.
Therefore, the $\bbk[V]^G$-module $\Ker(\phi_G)$ is also free, with the "same basis" $F_1,\dots,F_l$.
That is,  $\Ker(\phi)\simeq \Ker(\phi_G)\otimes_{\bbk[V]^G}\bbk[V]$.
\end{cl}
\begin{ex}   \label{ex:adjoint}
Let $\g$ be a semisimple Lie algebra of rank $l$. Then $\g\simeq \g^*$ as $G$-module,
$\ind\g=l$, and $\bbk[\g]^G=\bbk[f_1,\dots,f_l]$ is a graded polynomial algebra. Set $d_i=\deg f_i$.
Then $q(\g\md G)=\sum_{i=1}^l d_i=b(\g)$ is the dimension of a Borel subalgebra. Here
$\gig(G:\g)=\BT_l$ is a maximal torus. It is known that $\Mor_G(\g,\g)$ is a free 
$\bbk[\g]^G$-module generated by the differentials $\textsl{d}f_i=:F_i$, $i=1,\dots,l$. (This is a special case of a general result of Vust~\cite[Ch.\,III, \S\,2]{vust}, see also \cite[Theorem\,4.5]{p05}.) Here
$\deg F_i=d_i-1$ and hence Eq.~\eqref{eq:summa} holds. Thus,  
Theorems~\ref{thm:main1} and \ref{thm:main2} apply to $\g$ and $\q=\g\ltimes\g$. A specific feature of 
this case is that here $\phi_G\equiv 0$ and $\Ker(\phi_G)=\Mor_G(\g,\g)$.
\end{ex}

\begin{rmk}       \label{rmk:about-ss}
The semisimplicity of $G$ is assumed in Theorems~\ref{thm:main1} and \ref{thm:main2}, because Knop's 
results in \cite{kn86} heavily rely on this assumption. Using those results and Eq.~\eqref{eq:summa}, we
then prove that $\bigwedge_{i=1}^l F_i(v)\ne 0$ for all $v\in V_{\sf reg}$ and so on\dots 
\ But, if one can directly verify that
$Z=\{v\in V\mid \bigwedge_{i=1}^l F_i(v)=0\}$ does not contain divisors, 
then the proof of Theorem~\ref{thm:main1}(iii),(iv) goes through with $V\setminus Z$ in place of 
$V_{\sf reg}$ and without the semisimplicity condition.
(See Example~\ref{gl-nk} below.)
\\ \indent
Furthermore, if we know somehow that $\{F_i\}$ are $G$-equivariant (i.e., $F_i\in\Ker(\phi_G)$), then
$F_i(v)\in (\g_v)^{G_v}$ for all $v\in V$. For $v\in (V\setminus Z)\cap V_{\sf reg}$, this implies that $\dim \g_v=
\dim (\g_v)^{G_v}$.  Hence a generic stabiliser is abelian and the \ctc for $(G:V)$ implies that for $(Q:\q^*)$, cf. Theorem~\ref{thm:main2}(i). In this situation, we also have 
$\q^*\md Q\simeq V\md G\times \mathbb A^l$, and $\{F_1,\dots, F_l\}$ is a basis for both $\Ker(\phi)$ 
and $\Ker(\phi_G)$.
\end{rmk}
\begin{rmk}     \label{rmk:essential-h}
The assumptions of Theorem~\ref{thm:main2} that the adjoint representation of $H=\gig(G:V)$ is 
trivial and that $H$ is not contained in a proper normal subgroup of $G$ are essential. We will see in Example~\ref{ex:tri-SL} that  if this is not the case, then the 
morphisms $F_1,\dots,F_l$ satisfying \eqref{eq:summa} can be not $G$-equivariant and 
$\langle F_1,\dots,F_l\rangle$ affords a nontrivial representation of (a simple factor of) $G$.
\end{rmk}
\noindent
On the other hand, if $l=\dim\gig(G:V)=1$, then the assumptions of both theorems can be simplified,
and one also obtains stronger results.

\begin{thm}   \label{thm:l=1}
Suppose that $G\subset GL(V)$ is semisimple, $\codim (V\setminus V_{\sf reg})\ge 2$, and
$\mathsf{g.s.}(\g:V)=\te_1$. As usual, $\q=\g\ltimes V^*$. Then
\begin{itemize}
\item[\sf (i)] \ The \ctc holds for $(Q:\q^*)$;
\item[\sf (ii)] \ $\bbk[\q^*]^{R_u(Q)}$ is freely generated by a basis of $V^*$ and one more polynomial
$\hat F$ such that $\deg\hat F=\dim V-q(V\md G)+1$. In particular, $\q^*\md R_u(Q)\simeq V\times\mathbb A^1$;
\item[\sf (iii)] \ $\q^*\md Q\simeq V\md G\times\mathbb A^1$;
\item[\sf (iv)] \ If\/ $\bbk[V]^G$ is a polynomial ring, then $\q$ satisfies the Kostant criterion;
\item[\sf (v)] \  Furthermore, if  $\pi_{V,G}: V\to V\md G$ is equidimensional and 
{\color{MIXT} $(\ast)$} each irreducible
component of $\N_G(V):=\pi_{V,G}^{-1}(\pi_{V,G}(0))$ contains a $G$-regular point, 
then $\pi_{\q^*,Q}:\q^*\to \q^*\md Q$ is also equidimensional and the enveloping algebra $\eus U(\q)$ is 
a free module over its centre $\eus Z(\q)$.
\end{itemize}
\end{thm}
\begin{proof}
Since $l=1$, we have $b(\q)=\dim V+1$. Here we need only one morphism
$F: V\to \wedge^{\dim\g-1}\g^*\simeq \g$ of degree $\dim V-q(V\md G)$ such that 
$0\ne F(v)\in \g_v$ for all $v\in V_{\sf reg}$. The existence of such a $G$-equivariant morphism follows 
from Knop's theory \cite{kn86}. As the morphism $F$ is $G$-equivariant and  $F\in\Ker(\phi)$, 
the corresponding polynomial $\hat F$ lies in $\bbk[\q^*]^Q$. Then the proofs of 
Theorems~\ref{thm:main1} and \ref{thm:main2} apply and yield parts {\sf (i)-(iv)}.

{\sf (v)} \ The equidimensionality of $\pi_{V,G}$ is equivalent to that $\dim\N_G(V)=\dim V-\dim V\md G$, see~\cite[Eq.\,(8.1)]{VP}.
And for the equidimensionality of $\pi_{\q^*,Q}$, it suffices to prove that 
\[
 \dim\N_Q(\q^*)=\dim \q-\dim\q^*\md Q=\dim V+\dim\g -\dim V\md G-1=\dim\N_G(V)+\dim\g-1,
\]
where $\N_Q(\q^*)=\pi_{\q^*,Q}^{-1}(\pi_{\q^*,Q}(0))$.
It follows from {\sf (iii)} that 
\[
   \N_Q(\q^*)=\{(\xi,v)\mid v\in \N_G(V) \ \& \ 
   \hat F(\xi,v)=\langle F(v),\xi\rangle=0\} .
\]
In other words, $\N_Q(\q^*)=\bigl(\g^*\times \N_G(V)\bigr)\cap
\{\hat F=0\}$.
Under assumption {\color{MIXT}$(\ast)$}, we have 
$\dim\N_Q(\q^*)=\dim \N_G(V)+\dim\g-1$, as required. Then
$\eus S(\q)=\bbk[\q^*]$ is a free $\eus S(\q)^Q$-module; and, by a standard deformation argument, this implies that $\eus U(\q)$ is a free module over $\eus Z(\q)\simeq \eus S(\q)^Q$.
\end{proof}

Note that if $\N_G(V)$ contains finitely many $G$-orbits, then $\pi_{V,G}$ is 
equidimensional~\cite[\S\,5.2]{VP} and hence condition {\color{MIXT}$(\ast)$} is satisfied. 
\begin{rema} 
If $l\ge 1$ and $\N_G(V)$ contains finitely many $G$-orbits, then there is a general criterion for the equidimensionality of $\pi_{\q^*,Q}$ in terms of the stratification of $\N_G(V)$ determined by the covariants $F_1,\dots, F_l$. Namely, \\[.5ex]
\centerline{
  $\N_Q(\q^*)=\{(\xi,v)\mid v\in \N_G(V) \ \& \    
   \langle F_i(v),\xi\rangle=0 \quad i=1,\dots,l\}$}
\\[.5ex]
and using the  projection $\N_Q(\q^*)\to \N_G(V)$, $(\xi,v)\mapsto v$, one proves that 
\\[.5ex]
\centerline{$\pi_{\q^*,Q}$ is equidimensional \ $\Longleftrightarrow$ \  
$\dim\g_v+ \dim \langle F_1(v),\dots,F_l(v)\rangle \ge 2l$}
\\[.5ex]
for any $v\in\N_G(V)$. However, this condition is not easily verified in specific examples, if $l > 1$.
If $(G:V)$ is the isotropy representation of a symmetric variety such that $\gig$ is a torus, then a version of this condition is verified in \cite[Sect.\,5]{coadj}.
\end{rema}

\section{The codimension-$2$ condition for representations} 
\label{sect:codim2}

\noindent
In this section,  we provide some sufficient conditions for the \ctc to hold for $(G:V)$.
A $G$-stable divisor $D\subset V$ is said to be {\it bad}, if 
$\max_{v\in D}\dim G{\cdot}v< \max_{v\in V}\dim G{\cdot}v$. 
That is, if
\[
     \min_{v\in D}\dim G_v> \min_{v\in V}\dim G_v .
\]
Hence the \ctc holds for $(G:V)$  if and only if $V$ contains no bad divisors.

\begin{prop}   \label{prop:codim2-EQ}
Suppose that $G$ is reductive, the action $(G:V)$ is stable, and $\N_G(V)$ contains finitely many 
$G$-orbits. Then the \ctc holds for $(G:V)$.
\end{prop}
\begin{proof}
Since $\N_G(V)$ has finitely many orbits,  $\pi_{V,G}: V\to  V\md G$ is equidimensional and each fibre of 
$\pi_{V,G}$ also has finitely many orbits~\cite[\S\,5.2, Cor.\,3]{VP}. Assume that
$D$ is a ($G$-stable) bad divisor in $V$. 
Then $\pi_{V,G}(D)$ is a proper (closed) subvariety of $V\md G$, see e.g.~\cite[Theorem\,1]{tainan}, and 
since $\pi_{V,G}$ is equidimensional, $\pi_{V,G}(D)$ is actually a divisor
in $V\md G$. Hence $\dim D\cap \pi_{V,G}^{-1}(\xi)=\dim  \pi_{V,G}^{-1}(\xi)$ for any $\xi\in \pi_{V,G}(D)$ 
and therefore $D\cap \pi_{V,G}^{-1}(\xi)$ contains $G$-regular elements. Hence $D$ cannot be bad.
\end{proof}

\begin{ex}[Vinberg's $\theta$-groups \cite{vi76}]  \label{theta-1}
Let $\vartheta$ be an automorphism of $\g$ of finite order $k$. If $\varsigma=\sqrt[k]1$ is primitive, then 
$\g=\textstyle \bigoplus_{i\in \BZ/k\BZ}\g_i$, where $\g_i$ is the eigenspace of $\vartheta$ corresponding 
to $\varsigma^i$. The above decomposition is also called a {\it periodic grading} of $\g$.
Here $\g_0$ is reductive and each $\g_i$ is a $\g_0$-module. If $G_0$ is the connected subgroup of $G$ 
with $\Lie G_0=\g_0$, then the linear group $G_0\to GL(\g_1)$ is called a $\vartheta$-{\it group}. A 
fundamental invariant-theoretic property of $\vartheta$-groups is that $\N_{G_0}(\g_1)$ contains finitely 
many $G_0$-orbits and $\bbk[\g_1]^{G_0}$ is a polynomial ring. If $k=2$, then $(G_0:\g_1)$ is always 
stable. There are also many interesting examples of stable $\vartheta$-groups, if $k\ge 3$, 
see e.g.~\cite[\S\,9]{vi76}. 
\end{ex}

\begin{ex}[reduced $\theta$-groups]  \label{theta-2}
Let $\g=\bigoplus_{i\in\BZ}\g(i)$ be a $\BZ$-grading. Then $\g(i)=\{x\in\g \mid [h,x]=ix\}$ for a unique
semisimple element $h\in\g(0)$. Let $\g(0)'$ be the orthocomplement to $\bbk h$ in $\g(0)$ w.r.t. the Killing 
form on $\g$ and  $G(0)'\subset G(0)$ the corresponding connected subgroups of $G$. 
Here the reductive group $G(0)$ has finitely many orbits in each $\g(i)$ with $i\ne 0$~\cite{vi76}, while there is a 
dichotomy for $G(0)'$-orbits. Either the $G(0)'$-orbits in $\g(1)$ coincide with the $G(0)$-orbits, or
$\dim\g(1)\md G(0)'=1$ and the $G(0)'$-orbits in $\N_{G(0)'}\g(1)$ coincide with the $G(0)$-orbits~ \cite[Theorem\,2.9]{dG-V-Y}. In the latter case, the action $(G(0)':\g(1))$ is also stable. The linear groups of the form 
$G(0)'\to GL(\g(1))$ are called {\it reduced $\vartheta$-groups}.
\end{ex}

In the following assertion $G$ is not necessarily reductive.
\begin{thm}  \label{thm:kosik1}
Let $G:V_1\oplus V_2=V$ be a reducible representation. Suppose that generic isotropy groups
$S_i=\gig(G:V_i)$, $i=1,2$, exist and the \ctc holds for both $(S_1:V_2)$ and $(S_2:V_1)$. 
Then the \ctc holds also for $(G:V)$.
\end{thm}
\begin{proof}
Assume that $D\subset V$ is a bad divisor. Consider the 
projections $p_i: D\to V_i$, {$i=1,2$}. 

\textbullet\quad If $p_1$ is dominant and $p_2$ is not, then $D=V_1\times D_2$, where $D_2\subset V_2$ is a $G$-stable divisor. Take a generic point $x_1^o\in V_1$ such that $G_{x_1^o}= S_1$. The fact that 
$D$ is bad means that 
\[
   \min_{x_2\in D_2}\dim(G_{x_1^o})_{x_2}=\min_{\bar x\in D} \dim G_{\bar x}>\min_{\bar v\in V} \dim G_{\bar v}=\min_{x_2\in V_2}\dim(G_{x_1^o})_{x_2} .
\]
That is, $D_2$ appears to be a bad divisor for $(S_1: V_2)$. Thus, this case is impossible.

\textbullet\quad If $p_2$ is dominant and $p_1$ is not, then $D=D_1\times V_2$ and the argument is
"symmetric".

\textbullet\quad If both $p_1,p_2$ are dominant, then we again can take a point $\bar x=(x_1,x_2)\in D$
such that $G_{x_1}=S_1$. Here $p_1^{-1}(x_1)=\{x_1\}\times D_2$ and the similar argument shows that
$D_2$ is a bad divisor for $(S_1:V_2)$.
\end{proof}

\begin{Not}   \label{notation-phi}
In specific examples and the tables in Appendix~\ref{sect:tables}, 
we identify the representations $V$ of semisimple groups with their highest weights, using the {\it multiplicative\/} notation and the
Vinberg--Onishchik numbering of the fundamental weights. For instance, if $\vp_1,\dots,\vp_n$ are the fundamental weights of a simple algebraic group $G$, then $V=\vp_1^2+\vp_{n-1}$ stands for the direct sum of two simple $G$-modules with highest weights $2\vp_1$ and $\vp_{n-1}$. If $G=G_1\times G_2\times\ldots$ is semisimple, then the fundamental weights of the first (resp. second) factor are denoted by $\{ \vp_i\}$ (resp. $\{ \vp'_i\}$) and so on\dots \ The dual $G$-module for $\psi$ is denoted by $\psi^*$.
We omit the index for the unique fundamental weight of $SL_2$.
\end{Not}

\begin{ex}    \label{ex:kosik1}
We provide several cases, where the last theorem allows us to check the 
{codimension}-$2$ condition.
\\ \indent 
${\mathsf 1^o}$.  $G=SL_n$, $V_1=\vp_1^2$, and $V_2=\vp_2$. Here $S_1=SO_n$ and 
$(S_1:V_2)$ is equivalent to the adjoint representation of $SO_n$ modulo a trivial summand.
If $n$ is even, then $S_2=Sp_{n}$ and $(S_2:V_1)$ is equivalent to the adjoint representation of 
$Sp_n$ modulo a trivial summand. This already shows that \ctc holds if $n$ is even. 
For $n$ odd, $S_2$ is not reductive and the only {\sf a priori} possible bad divisor is $D_1\times V_2$, where $D_1$ consists of the symmetric matrices with $\det=0$. Here a direct calculation of stabilisers shows that this divisor is not bad. Thus, the 
\ctc holds for all $n$.

${\mathsf 2^o}$.  $G=SL_n$, $V_1=\vp_1^2$, and $V_2=\vp_2^*=\vp_{n-2}$. This is similar to ${\mathsf 1^o}$.

${\mathsf 3^o}$.  $G=SL_n\times SL_n$ and $V_1=V_2=\vp_1\vp'_1$. Here $S_1=S_2=\Delta_{SL_n}\simeq SL_n$
and $(S_1: V_2)$ is equivalent to the adjoint representation of $SL_n$ modulo a trivial summand.

${\mathsf 4^o}$.  $G=Sp_6$ and $V_1=V_2=\vp_2$. Here $S_1=S_2=(SL_2)^3$ and 
and $(S_1: V_2)$ is equivalent to $(SL_2\times SL_2\times SL_2: \vp\vp'+\vp\vp''+\vp'\vp'')$
 modulo a 2-dimensional trivial summand. Applying Theorem~\ref{thm:kosik1} to the last representation, one readily obtains the \ctc.
\end{ex}

Below is a variation of Theorem~\ref{thm:kosik1} that concerns the case in which $V_1\simeq V_2$.

\begin{thm}   \label{thm:kosik2}
For any representation $G\to GL(V)$, one naturally defines the representation of $\hat G=G\times SL_2$ in 
$\hat V=V\otimes \bbk^2$. Let $G_\ast$ be a generic isotropy group for $(G:V)$. If \ctc holds for 
$(G_\ast: V)$ and $\mathsf{g.s.}(\g:V)=\mathsf{g.s.}(\hat \g:\hat V)$, then 
\ctc also holds for $(\hat G:\hat V)$. 
\end{thm}
\begin{proof}
Since $\hat V\vert_{G}=V\oplus V$,  Theorem~\ref{thm:kosik1} shows that the \ctc holds for
$(G:V\oplus V)$.  Let $\hat D\subset \hat V$ be a  $\hat G$-stable divisor. As above, 
consider the $G$-equivariant projections $p_i:\hat D\to V^{(i)}$, where $V^{(i)}$ is the $i$-th copy of $V$ and $i=1,2$. Since $\hat D$ 
is $SL_2$-stable, both projections must be dominant. Take $(x_1,x_2)\in \hat D$ such that  
$G_{x_1}=G_\ast$. Since $x_1$ is $G$-generic and \ctc holds for $(G_\ast: V)$,  there is $x_2\in p_2(p_1^{-1}(x_1))$ such that
\[
  \dim (G_{x_1})_{x_2}=\min_{v_2\in V^{(2)}}\dim (G_{x_1})_{v_2}=\min_{\bar v\in V\oplus V}\dim G_{\bar v}=
  \min_{\hat v\in \hat V} \dim G_{\hat v} .
\]
This means that $\hat D$ cannot be bad.
\end{proof}

\begin{ex}   \label{ex:plus-SL2}
Theorem~\ref{thm:kosik2} applies, if we add the factor $SL_2$ to $G$ in Example~\ref{ex:kosik1},
${\mathsf 3^o}$ \& ${\mathsf 4^o}$.
\end{ex}

\begin{thm}    \label{thm:kosik3}
Suppose that the  \ctc holds for $(G_1\times G_2: V_1\otimes V_2=V)$ and 
$\mathsf{g.s.}(\g_1\times \g_2: V_1\otimes V_2)=\mathsf{g.s.}(\g_1: V_1^{\oplus d})$, where 
$d=\dim V_2$. Then \ctc also holds for $(G_1: V_1^{\oplus d})$.
\end{thm}
\begin{proof}
Assume that $D\in V$ is a bad divisor for $(G_1{:} V_1^{\oplus d})$. 
Then $\dim G_1{\cdot}u< \max_{v\in V}\dim G_1{\cdot}v$  for all $u\in D$. 
The coincidence of generic stabilisers implies that $D$ is also $G_1\times G_2$-stable and then 
\[
  \dim G{\cdot}u< \max_{v\in V}\dim G_1{\cdot}v +\dim G_2=\max_{v\in V}\dim G{\cdot}v \quad \text{ for all } u\in D .
\]
Hence $D$ is bad for $(G:V)$, too. A contradiction!
\end{proof}
\begin{ex}   \label{ex:perexod}
The representation $(SL_6\times SL_3: \vp_2\vp'_1)$ is the $\vartheta$-group associated with an 
automorphism of order $3$ of $\GR{E}{7}$, see item~5 in the table in \cite[\S\,9]{vi76}. A generic isotropy 
group $H$ here is reductive (namely,  $\Lie H=\te_1$). Therefore, this action is stable and hence \ctc holds 
here (use Prop.~\ref{prop:codim2-EQ}). All assumptions of Theorem~\ref{thm:l=1} are satisfied here, and therefore $\q=(\mathfrak{sl}_6{\times}\mathfrak{sl}_3)\ltimes (\vp_2\vp'_1)^*$ satisfies the Kostant criterion and $\eus U(\q)$ is a free module over $\eus Z(\q)$.
\\ \indent
Forgetting about $SL_3$, we obtain the representation $(SL_6: 3\vp_2)$. Since both have the same
generic stabilisers (namely $\te_1$), the \ctc also holds for the latter in view of Theorem~\ref{thm:kosik3}.
Here the algebra $\bbk[3\vp_2]^{SL_6}$ is still polynomial~\cite{ag79,gerry1}, but the equidimensionality 
of the quotient morphism fails~\cite{gerry2}. Hence $\q'=\mathfrak{sl}_6\ltimes 3\vp_2^*$ satisfies the 
Kostant criterion, but $\eus U(\q')$ is {\bf not} a free $\eus Z(\q')$-module.
\end{ex}

\section{Constructing covariants for semi-direct products, I}  
\label{sect:primery}

\noindent
If an action $(G:V)$ is associated with a periodic or $\BZ$-grading of a simple Lie algebra, 
then usually most of the assumptions of Theorems~\ref{thm:main1} and \ref{thm:main2} are automatically 
satisfied for it. The most appealing and non-trivial task is to produce linearly independent morphisms 
$\{F_i\}$ in $\Ker(\phi)$ such that \eqref{eq:summa} holds.

\begin{ex}   \label{ex:tri-SL} $G=SL(V_1){\times} SL(V_2){\times} SL(V_3)$ and  
$V=V_1{\otimes} V_2{\otimes} V_3$, where $\dim V_1=\dim V_2=n$ and $\dim V_3=2$. In other words,
$G=SL_n{\times} SL_n{\times}SL_2$ and $V=\vp_1\vp'_1\vp''\simeq \bbk^n\otimes\bbk^n\otimes \bbk^2$.
\\ \indent Upon the 
restriction to $\tilde G:=SL(V_1)\times SL(V_2)$, the space $V$ splits in two copies of $V_1\otimes V_2$. We regard the $\tilde G$-module $V_1\otimes V_2$ as the space $n$ by $n$ matrices, 
equipped with the action $(g_1,g_2){\cdot}A=g_1 Ag_2^{-1}$, where $g_i\in SL(V_i)$. The 
corresponding action of $(s_1,s_2)\in\tilde \g$ is given by $(s_1,s_2){\cdot}A=s_1 A-As_2$. We think of elements
of $V$ as pairs $(A,B)$ of $n$ by $n$ matrices, where the action of 
$\left(\begin{array}{cc} \ap & \beta \\ \gamma & \delta \end{array}\right)\in SL_2=SL(V_3)$ is given by 
$(A,B)\mapsto (\ap A+\beta B, \gamma A+\delta B)$. By Examples~\ref{ex:kosik1}($3^o$) and \ref{ex:plus-SL2}, 
the \ctc holds for both $(G:V)$ and $(\tilde G:V)$.
The algebra $\bbk[V]^{\tilde G}$ is polynomial and its
basic invariants are the coefficients of the characteristic polynomial 
\[
   \det (A+\lb B)=\sum_{i=0}^n f_i(A,B)\lb^i, \text{ where  $f_0(A,B)=\det A$ and $f_n(A,B)=\det B$},
\]
see  e.g.~\cite[Theorem\,4]{pervushin02}.
Since $\deg f_i(A,B)=n$ for all $i$, $q(V\md \tilde G)=n(n+1)$. Looking at the weights of the 
polynomials $\{f_i(A,B)\}_{i=0}^n$ w.r.t. a maximal torus in $SL_2$, one realises that 
$V\md\tilde G$ is isomorphic to $(\vp'')^n$ (the space of binary forms of degree $n$) as an $SL_2$-module. 
(We also write $\sfr_n$ for this $SL_2$-module.) It is known that  
$q(\sfr _n\md SL_2)=\dim \sfr _n=n+1$ for $n\ge 3$. In our case, the coordinates in $\sfr _n=V\md \tilde G$ 
are of degree $n$ w.r.t. the initial grading of $\bbk[V]$. Therefore, 
\[
    q(V\md G)=n{\cdot}q(\sfr _n\md SL_2)= n(n+1)=q(V\md \tilde G) \ \text{ if } \ n\ge 3 . 
\]
It is easily verified that $\tilde H:=\gig(\tilde G:V)\simeq \BT_{n-1}$ for any $n\ge 2$, where the torus 
$\BT_{n-1}$ is diagonally embedded in $\tilde G\simeq SL_n\times SL_n$. Furthermore, the identity 
component of $H=\gig(G:V)$ is the same torus for $n\ge 3$. In other words, 
$\h=\mathsf{g.s.}(\g{:}V)=\mathsf{g.s.}(\tilde\g{:}V)=\tilde \h$ for $n\ge 3$. (See Example~\ref{ex:cubic-matr} for $(G:V)$ with $n=2$.)
However, $H$ can be disconnected. Using the isomorphism $V\md \tilde G\simeq \sfr_n$, one verifies that $H/H^0$ is isomorphic to $\gig(SL_2:\sfr_n)$, and the latter is isomorphic to 
\\ \indent
{\color{MIXT}\textbullet} \  $\BZ_3,$ if $n=3$; \ {\color{MIXT}\textbullet} \  $\BZ_2\ltimes\BZ_4,$ if $n=4$; 
\ {\color{MIXT}\textbullet} \  $\{1\},$ if $n\ge 5$ is odd; \ {\color{MIXT}\textbullet} \  $\BZ_2,$ if $n\ge 6$ is even.

We will compare below the coadjoint representations of the Lie algebras $\q=\g\ltimes V^*$ and 
$\tilde\q=\tilde\g\ltimes V^*$ for $n\ge 3$.
Accordingly, we consider the corresponding connected groups $Q$ and $\tilde Q$,  
two morphisms of $\bbk[V]$-modules
\[
    \phi:  \Mor(V,\g) \to \Mor(V,V) \ \text{ and } \ \tilde\phi:  \Mor(V,\tilde\g) \to \Mor(V,V) 
\]
and the corresponding morphisms $\phi_G$ and $\tilde\phi_{\tilde G}$ of modules of covariants (see 
Section~\ref{sect:rank}). Clearly, $\Mor(V,\tilde\g)\subset \Mor(V,\g)$ and 
$\tilde\phi=\phi\vert_{\Mor(V,\tilde\g)}$. Note also that $R_u(Q)=R_u(\tilde Q)$.

For $A\in\gln$, let $A^*$ denote the {\it adjugate\/} of $A$, i.e., the transpose of its cofactor matrix. (Hence 
$AA^*=A^*A=(\det A)I$.)  Note that $A\mapsto A^*$ is a polynomial mapping of degree
$n-1$.  Let $A\mapsto \bar A=A-\frac{\tr(A)}{n}I$ denote the projection from $\gln$ to $\sln$. 

Consider the morphism $F\in\Mor(V,\tilde\g)$, where
$F(A,B)=(\ov{BA^*}, \ov{A^*B})\in\tilde\g\subset\g$. Here $\ov{BA^*}$ (resp. $\ov{A^*B}$) is regarded as 
an element of $\mathfrak{sl}(V_1)$ (resp. $\mathfrak{sl}(V_2)$).
One readily verifies that $F(A,B){\cdot}(A,B)=0$, cf. the proof of 
Theorem~\ref{thm:polarisation-para-matric}(i). Hence $F\in \Ker(\tilde\phi)\subset \Ker(\phi)$.
Since the map $A\mapsto A^*$ has degree $n-1$, we obtain 
$\deg F=n$.  We will see below that the morphism $F$ is $\tilde G$-equivariant.
However, it is not $SL_2$-equivariant, hence not $G$-equivariant. Still, $F$ is a lowest weight vector in
a simple $SL_2$-module $\sfr_{n-2}$. Indeed, for any $\gamma$ we have 
\[
   F(A,\gamma A+B)=(\ov{(\gamma A+B)A^*}, \ov{A^*(\gamma A+B)})=
   \ov{\gamma AA^*+BA^*}, \ov{\gamma A^*A+A^*B})=F(A,B) ,
\]
i.e.,  the subgroup $\{\begin{pmatrix} 1 & 0\\ \gamma  & 1\end{pmatrix}\mid \gamma \in \bbk\}\subset SL_2$ stabilises  $F$. 
By a direct calculation, we also have
$g\ast F=t^{2-n}F$  for $g=\begin{pmatrix} t & 0\\ 0 & t^{-1} \end{pmatrix}$.
\\ \indent
Having at hand one suitable covariant, we perform a ``polarisation''. Consider
\[
   \mathbb F_\lb(A,B):=F(A+\lb B, B)=(\ov{B(A+\lb B)^*}, \ov{(A+\lb B)^*B})=
   \sum_{i=0}^{n-1} F_i(A,B)\lb^i .
\]
Note that $F_0=F$ and $F_{n-1}(A,B)=(\ov{BB^*}, \ov{B^*B})=0$. That is, we obtain only the 
morphisms $F_0,F_1,\dots,F_{n-2}$ in $\Mor(V,\tilde \g)$. It follows from the previous observation 
that the $\bbk$-linear span $\langle F_0,F_1,\dots,F_{n-2}\rangle$ is an $SL_2$-module 
isomorphic to $\sfr_{n-2}$.

\begin{thme}   \label{thm:polarisation-para-matric} We have 
\begin{itemize}
\item[\sf (i)]  \ $\mathbb F_\lb$ is a $\tilde G$-equivariant morphism for any $\lb\in\bbk$. Therefore,
all $\{F_i\}$ are $\tilde G$-equivariant;
\item[\sf (ii)] \ $F_i\in \Ker(\tilde\phi)$ for all $i$.
\end{itemize} 
\end{thme}
\begin{proof}
(i)  By definition,
\[
  \mathbb F_\lb(g_1Ag_2^{-1}, g_1Bg_2^{-1})=\bigl(\ov{g_1Bg_2^{-1}{\cdot}(g_1(A+\lb B)g_2^{-1})^*}, 
  \ov{(g_1(A+\lb B)g_2^{-1})^*{\cdot}g_1Bg_2^{-1}}\bigr).
\]
If $A+\lb B$ is invertible, then the first component is being transformed as follows:
\begin{multline*}
  \ov{g_1Bg_2^{-1}(g_1(A+\lb B)g_2^{-1})^*}=\det(A+\lb B)\ov{g_1Bg_2^{-1}{\cdot}(g_1(A+\lb B)g_2^{-1})^{-1}}\\
  =\det(A+\lb B)\ov{g_1B(A+\lb B)^{-1}g_1^{-1}}=g_1\ov{B(A+\lb B)^*}g_1^{-1} .
\end{multline*}
Likewise, for the second component, we obtain $g_2\bigl(\ov{(A+\lb B)^*B}\bigr)g_2^{-1}$.
Thus, \\[.3ex]
\centerline{$\mathbb F_\lb((g_1,g_2){\cdot}(A,B))=
\mathbb F_\lb(g_1Ag_2^{-1}, g_1Bg_2^{-1}) =(g_1,g_2){\cdot}
\mathbb F_\lb(A,B)$ }
\vskip.5ex\noindent
whenever $A+\lb B$ is invertible.
Since $\mathbb F_\lb$ is a polynomial mapping that is $\tilde G$-equivariant on the open subset of triples
$(A,B,\lb)$ such that $A+\lb B$ is invertible, it is always equivariant.

(ii) It suffices to verify that  $\mathbb F_\lb(A,B){\cdot}(A,B)=0$ for any $\lb$. The first component in the 
LHS equals
\begin{multline}  \label{eq:difference}
   \ov{B(A{+}\lb B)^*}A-A\ov{(A{+}\lb B)^*B} \\  
   = {B(A{+}\lb B)^*}A-\frac{\tr B(A{+}\lb B)^*}{n}A-A{(A{+}\lb B)^*B}+\frac{\tr(A{+}\lb B)^*B}{n}A .
\end{multline}
Now, if both $A$ and  $A{+}\lb B$ are invertible, then
\begin{multline*}
   {B(A+\lb B)^*}A=\det(A{+}\lb B){\cdot}B(A{+}\lb B)^{-1}A=\det(A{+}\lb B){\cdot}B(A(I{+}A^{-1}B))^{-1}A\\
   =\det(A{+}\lb B){\cdot}B(I{+}A^{-1}B)^{-1}=\det(A{+}\lb B){\cdot}(B{-}\lb BA^{-1}B{+}\lb^2 BA^{-1}BA^{-1}B{-}\cdots) .
\end{multline*}
A similar transform yields the very same formula for $A{(A{+}\lb B)^*B}$. Since the difference in 
\eqref{eq:difference} vanishes on the open subset of triples $(A,B,\lb)$, where $A$ and  $A{+}\lb B$ are invertible, it is identically zero. And likewise for the second component in $\mathbb F_\lb(A,B){\cdot}(A,B)$.
\end{proof}

\begin{rema} Permuting $A$ and $B$ in the definition of $F=F_0$, one defines the companion morphism
$\hat F\in\Mor(V,\tilde\g)$ by $\hat F(A,B)=(\ov{AB^*}, \ov{B^*A})$. Then we can prove that 
$\hat F=-F_{n-2}$.
\end{rema}

Note that $\sum_{i=0}^{n-2} \deg F_i=n(n-1)=\dim V-q(V\md \tilde G)=\dim V-q(V\md G)$.
Hence $\tilde G, V,\tilde\q,$ and the covariants $F_0,\dots,F_{n-2}$ satisfy all the assumptions of Theorems~\ref{thm:main1} and \ref{thm:main2}. Hence 
{\sl \begin{itemize}
\item \ $\tilde \q^*\md R_u(\tilde Q)\simeq V\times \mathbb A^{n-1}$ and 
   $\tilde \q^*\md \tilde Q\simeq V\md \tilde G\times \mathbb A^{n-1}\simeq \mathbb A^{2n}$;
\item \ $\Ker(\tilde\phi)$ (resp. $\Ker(\tilde\phi_{\tilde G})$) is a free $\bbk[V]$
(resp. $\bbk[V]^{\tilde G}$)  -module  with basis $F_0,\dots,F_{n-2}$;
\item \ the Kostant criterion holds for 
$\tilde \q=(\sln\times\sln) \ltimes (\bbk^n\otimes\bbk^n\otimes \bbk^2)^*$. 
\end{itemize}
}
\noindent
However, $G,V,$ and $\q=(\sln\times\sln\times\tri) \ltimes (\bbk^n\otimes\bbk^n\otimes \bbk^2)^*$ do 
{\bf not} satisfy all the assumptions of Theorem~\ref{thm:main2}. For, either 
$\h\ne \h^H$ ($n=3,4$) or $H$ is contained in a proper normal subgroup of $G$ ($n\ge 5$).
But Theorem~\ref{thm:main1} still applies, and we have
$\q^*\md R_u(Q)\simeq V\times \mathbb A^{n-1}$. Then 
$\q^*\md \tilde Q\simeq V\md\tilde G\times \mathbb A^{n-1}$, and the last variety is isomorphic to 
$\sfr_{n}\oplus \sfr _{n-2}$ as $Q/\tilde Q$-module, i.e., $SL_2$-module. Therefore, 
$\q^*\md Q\simeq (\sfr_n\oplus \sfr_{n-2})\md SL_2$, which is not an affine space for $n\ge 3$. In other 
words, $\bbk[\q^*]^Q$ is not a polynomial ring for $n\ge 3$. For instance, it is a 
hypersurface for $n=3,4$, see e.g.~\cite[3.4.3]{springer}. 

\begin{rme}   \label{rmk:sravnenie-ker}
Since $\mathsf{g.s.}(\g:V)=\mathsf{g.s.}(\tilde\g:V)=\te_{n-1}$, we have
$\rk\Ker(\phi)=\rk\Ker(\tilde\phi) =n-1$ by Eq.~\eqref{eq:rk-ker-hat-phi}.  
Moreover, because $\tilde H=\gig(\tilde G:V)$ is abelian and connected, we also get
$\rk\Ker(\tilde\phi)=\rk\Ker(\tilde\phi_{\tilde G})$. 
\\ \indent But the situation for $\phi$ and $\phi_G$ is different.
If $n=3,4$, then the component group
$H/\tilde H$ acts nontrivially on $\h$ and, actually, $\h^H=\{0\}$. Therefore, $\rk\Ker(\phi_{G})=0$.
On the other hand, if $n\ge 5$, then $\h^H=\h$, hence $\rk\Ker(\phi)=\rk\Ker(\phi_{G})$. 
However, even if $\Ker(\phi)$ and $\Ker(\phi_{G})$ have the same rank, the free generators of the former are not $G$-equivariant (they are only $\tilde G$-equivariant). In fact, we do not know the generators of the $\bbk[V]^G$-module
$\Ker(\phi_{G})$ if $n\ge 5$.
\end{rme}
\end{ex}

\begin{ex}   \label{ex:cubic-matr}
The case of $n=2$ in Example~\ref{ex:tri-SL} does not fit into the general picture with $n\ge 3$, so we 
consider it separately. Now $G=(SL_2)^3$ and $V=\vp\vp'\vp''$. This is a reduced $\vartheta$-group
(see Example~\ref{theta-2}) related to a $\BZ$-grading of $\GR{D}{4}$. Therefore $\ctc$ holds here.
We have $V\md G=\mathbb A^1$, $q(V\md G)=4$, and $\gig(G:V)\simeq \BT_2$. More precisely,
if the elements of a maximal torus 
\[
   \BT=\left\{\begin{pmatrix} t_1 & 0 \\ 0 & t_1^{-1} \end{pmatrix} ,
   \begin{pmatrix} t_2 & 0 \\ 0 & t_2^{-1} \end{pmatrix} ,
   \begin{pmatrix} t_3 & 0 \\ 0 & t_3^{-1} \end{pmatrix} \mid t_i\in \bbk^{\times} 
   \right\}\subset G
\]
are represented as triples $(t_1,t_2,t_3)$ , then $\gig(G:V)=
\{(t_1,t_2,t_3)\mid t_1t_2t_3=1\}$.
\\ \indent
The elements of $V$ can be regarded as cubic $2$-matrices with entries $a_{ijk}$, see Fig.~\ref{fig1}, where the $i$-th factor of $G$ acts along the $i$-th coordinate, $i=1,2,3$.
\begin{figure}[htbp]
\setlength{\unitlength}{0.035in}
\caption{A cubic 2-matrix}   \label{fig1}
\begin{center}
\begin{picture}(60,60)(0,0)
\put(-15,20){$M=$}   
\put(65, 20){$\in V=\vp\vp'\vp''$}
\multiput(0,0)(35,0){2}{\circle{10}}
\multiput(0,35)(35,0){2}{\circle{10}}
\multiput(15,15)(35,0){2}{\circle{10}}
\multiput(15,50)(35,0){2}{\circle{10}}

\multiput(5,0)(0,35){2}{\line(1,0){25}}
\multiput(20,15)(0,35){2}{\line(1,0){25}}
\multiput(0,5)(35,0){2}{\line(0,1){25}}
\multiput(15,20)(35,0){2}{\line(0,1){25}}

\multiput(3.5,3.5)(35,0){2}{\line(1,1){8}}
\multiput(3.5,38.5)(35,0){2}{\line(1,1){8}}

\put(-3.3,-.5){\small $a_{000}$}
\put(31.7,-.5){\small $a_{100}$}
\put(-3.3,34.5){\small $a_{001}$}
\put(31.7,34.5){\small $a_{101}$}
\put(11.7,14.5){\small $a_{010}$}
\put(46.7,14.5){\small $a_{110}$}
\put(11.7,49.5){\small $a_{011}$}
\put(46.7,49.5){\small $a_{111}$}
\end{picture}
\end{center}
\end{figure}
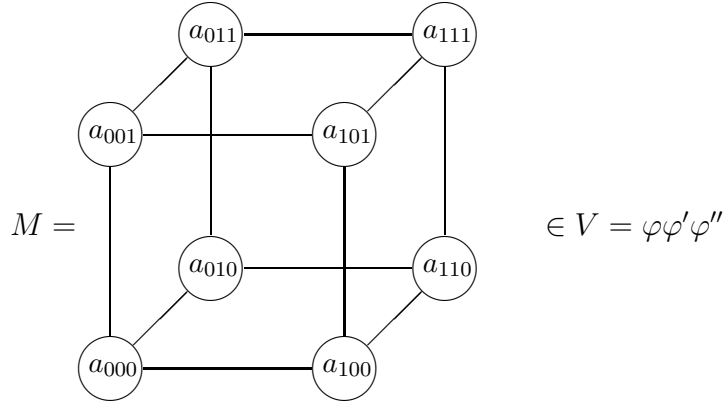

We provide below three morphisms from $V$ to $\tri$ that are thought of as morphisms to the consecutive 
factors of $\g$, where the column $\left[\begin{smallmatrix} m \\ n \\ p\end{smallmatrix}\right]$ represents 
the matrix $\bigl(\begin{smallmatrix} n & m \\ p & -n \end{smallmatrix}\bigr)$:
\begin{gather*}   
  \tilde F_1(M)=
  \begin{bmatrix} a_{111}a_{100}-a_{101}a_{110} \\
  a_{111}a_{000}+a_{011}a_{100} -a_{001}a_{110}-a_{101}a_{010}\\
  a_{011}a_{000}-a_{001}a_{010} 
  \end{bmatrix}  \\
  \tilde F_2( M)=
  \begin{bmatrix} a_{111}a_{010}-a_{011}a_{110} \\
  a_{111}a_{000}-a_{011}a_{100} -a_{001}a_{110}+a_{101}a_{010}\\
  a_{101}a_{000}-a_{001}a_{100} 
  \end{bmatrix}  \\
  \tilde F_3(M)=
  \begin{bmatrix} a_{111}a_{001}-a_{101}a_{011} \\
  a_{111}a_{000}-a_{011}a_{100} +a_{001}a_{110}-a_{101}a_{010}\\
  a_{110}a_{000}-a_{100}a_{010} 
  \end{bmatrix}
\end{gather*}
Letting $\mathbb F_{\lb,\mu,\nu}(M){=}(\lb \tilde F_1(M), \mu \tilde F_2(M), \nu \tilde F_3(M))$ with 
$\lb,\mu,\nu\in\bbk$, we 
obtain a 3-dimensional subspace of $\Mor(V,\g)$, and one verifies directly that 
$\mathbb F_{\lb,\mu,\nu}\in \Ker(\phi)$ if and only if $\lb+\mu+\nu=0$.
Then $F_1=\mathbb F_{\lb,-\lb,0}$ and $F_2=\mathbb F_{0,\mu,-\mu}$ satisfy~\eqref{eq:summa} and 
Theorems~~\ref{thm:main1} and \ref{thm:main2} apply. 
\\ \indent {\sl Hence $\Ker(\phi)$ (resp. $\Ker(\phi_G)$)
is a free\/ $\bbk[\vp\vp'\vp'']$-module (resp. $\bbk[\vp\vp'\vp'']^{(SL_2)^3}$-module) and $\q=(\tri)^3\ltimes \vp\vp'\vp''$ satisfies the Kostant 
criterion. Furthermore, using the explicit classification of $G$-orbits in $V$, one can prove that $\ctrc$ holds 
for $(G:V)$ and hence for $(Q:\q^*)$,  and also that
$\eus U(\q)$ is a free $\eus Z(\q)$-module.}
\end{ex}     

\begin{ex}  \label{gl-nk}
$G=\prod_{i=1}^k GL(\BU_i)$ and $V=\bigoplus_{i=1}^k \BU_i\otimes \BU_{i+1}^*$, where $\BU_{k+1}=\BU_1$.
\\ 
Assume that $\dim\BU_i=n$ for all $i$. Then $(G:V)$ is a $\vartheta$-group related to an 
automorphism of order $k$ of $\tilde\g=\mathfrak{gl}(\BV)=\mathfrak{gl}_{nk}$, where $V=\BU_1\oplus\dots\oplus\BU_k$. Namely, if $\varsigma=\sqrt[k]1$ and 
\[
  t=\mathsf{diag}(\underbrace{\varsigma^{k-1},\dots,\varsigma^{k-1}}_n, \dots,
  \underbrace{\varsigma,\dots,\varsigma}_n, \underbrace{1,\dots,1}_n 
   ) ,
\]
then $\vartheta=\Ad(t)$, $G=\tilde G_0$, and $V=\tilde \g_1$. In the matrix form, we have \\
\[ \g=\begin{pmatrix} \mathfrak{gl}(\BU_1) & 0 & \cdots & 0 \\
 0 & \mathfrak{gl}(\BU_2) &  & \\
  \vdots& & \ddots & \\
 0 &  & & \mathfrak{gl}(\BU_k) 
 \end{pmatrix},  \ \text{ and } \  
 \eus M=\begin{pmatrix} 0 & M_1 & \cdots & 0\\
 \cdots & 0 & \ddots & & \\
   & & \ddots & M_{k-1}\\
 M_k &  & \cdots & 0 
 \end{pmatrix} 
\]
is a typical element of $V=\tilde\g_1$. We also write $\eus M\doteqdot (M_1,\dots,M_k)$. 
Here $\dim\g=kn^2=\dim V$,  $\mathsf{g.s.}(\g:V)=\te_n$, and 
$V\md G\simeq \mathbb A^n$. The centre of $\tilde G=GL(\BV)$ belongs to $G$ and acts trivially on everything. Therefore, without any harm, we can replace $\tilde\g=\mathfrak{gl}_{nk}$ with 
$\mathfrak{sl}_{nk}$. But, it is notationally simpler to deal with $\mathfrak{gl}_{nk}$.

If $g_i\in GL(\BU_i)$, $\boldsymbol{g}=(g_1,\dots,g_k)\in G$, and $\eus M\doteqdot (M_1,\dots,M_k)$, then
the $G$-action on $V$ is given by 
\[   
   \boldsymbol{g}{\cdot}\eus M\doteqdot(g_1M_1g_2^{-1},g_2M_2g_3^{-1},\dots, g_kM_kg_1^{-1}) .
\]
Accordingly, for $\boldsymbol{s}=(s_1,\dots,s_k)\in \g$, we have
\beq     \label{eq:s:M}
   \boldsymbol{s}{\cdot}\eus M\doteqdot(s_1M_1-M_1s_2,\ s_2M_2-M_2s_3,\dots,\ s_k M_k-M_k s_1) .
\eeq
Vinberg's theory (Example~\ref{theta-1}) implies that here $\bbk[V]^G$ is a polynomial ring and 
$\N_G(V)$ contains finitely many $G$-orbits. But in this case, one can explicitly describe the basic invariants and thereby compute $q(V\md G)$. The representation $(G:V)$ is a quiver representation related
to the extended Dynkin quiver $\GRt{A}{nk-1}$, and the algebra $\bbk[V]^G$ is well known. But we prefer an
"elementary" invariant-theoretic point of view in our exposition.
\setcounter{rmke}{0}
\begin{thme}   \label{thm:gl-nk1}
The algebra $\bbk[V]^G$ is freely generated by the coefficients of the characteristic polynomial of the
matrix $M_1\cdots M_k$ (or any cyclic permutation of this product). In particular, the degrees of the basic invariants are $k,2k,\dots,nk$ and $\dim V-q(V\md G)= k\genfrac{(}{)}{0pt}{}{n}{2}$.
\end{thme}
\begin{proof}
Using the {\sf First Fundamental Theorem of Invariant Theory} or the
Igusa lemma~\cite[Theorem\,4.12]{VP}, one readily verifies that the quotient of $V$ by $G'=\prod_{i=2}^k GL(\BU_i)$ is given by the mapping $\eus M\mapsto M_1\cdots M_k\in \mathsf{Mat}_n(\bbk)$. Since
$g_1{\cdot}(M_1\cdots M_k)=g_1M_1\cdots M_kg_1^{-1}$, the induced action of $GL(\BU_1)=G/G'$ on 
$V\md G'\simeq \mathsf{Mat}_n(\bbk)$ is equivalent to the adjoint representation.
\end{proof}
Define the morphism $F_i\in \Mor(V,\g)$ by $F_i(\eus M)=\eus M^{ki}$ (the usual matrix power in
$\mathfrak{gl}_{nk}$).

\begin{thme}   \label{thm:gl-nk2}
We have 
\begin{itemize}
\item[\sf (i)] \ each $F_i$ is $G$-equivariant, lies in $\Ker(\phi)$, and $\sum_{i=0}^{n-1}\deg F_i=k\genfrac{(}{)}{0pt}{}{n}{2}$;
\item[\sf (ii)] \  For $Z=\{ \eus M\in V\mid \bigwedge_{i=0}^{n-1}F_i(\eus M)=0\}$, we have
$\codim_VZ\ge 2$. 
\end{itemize}
\end{thme}
\begin{proof}
(i) It is clear from the definition that all $F_i$ are $G$-equivariant. Next, $\eus M^k$ is a block-diagonal 
matrix, where the first block is $M_{[1,k]}:=M_1\cdots M_k$ and the subsequent blocks are cyclic permutations of this product. The equality $F_1(\eus M){\cdot}\eus M=0$ readily
follows from this observation and \eqref{eq:s:M}. And likewise for $F_i$ ($i\ge 2$). The case of $i=0$ is obvious.
\\ \indent
(ii) We have the commutative diagram
\[
\xymatrix{ V\ar[r]^{\pi_{G'}}\ar[rd]^{\pi_G} & V\md G'\ar[d]^{\pi_{G/G'}}  \\ 
& V\md G\simeq \mathbb A^n} 
\qquad 
\xymatrix{ \eus M\ar@{|->}[r] & M_{[1,k]}\ar@{|->}[d]   \\ 
& \bigl(\sigma_1 (M_{[1,k]}),\dots, \sigma_n(M_{[1,k]})\bigr)
}
\]
If $M_{[1,k]}$ is a $G/G'$-regular (=\,{\sl non-derogatory}) matrix, then $\{F_i(\eus M)\}_{i=0}^{n-1}$ are 
linearly independent. Let $Y$ denote the variety of all derogatory matrices in $\mathsf{Mat}_n(\bbk)$.
Then $Z\subset \pi_{G'}^{-1}(Y)$, and it suffices to prove that $\codim \pi_{G'}^{-1}(Y)\ge 2$. Consider the matrices
$\eus M(1)\doteqdot (I_n,\dots,I_n,A)$ and $\eus M(2)\doteqdot (I_n,\dots,I_n,E,I_n)$, where
$A=\mathsf{diag}(a_1,\dots,a_n)$ with $a_i\ne a_j$ and $E={\small \begin{pmatrix} 0 & 1 & \cdots & 0\\
 \cdots & 0 & \ddots & & \\
   & & \ddots & 1\\
 0 &  & \cdots & 0 
\end{pmatrix}}$ is a regular nilpotent element of $\gln$. The plane
$\mathcal P=\{\ap\eus M(1)+\beta\eus M(2)\mid \ap,\beta\in\bbk\}$ 
has the property that, for any nonzero $\eus M\in \mathcal P$, the corresponding matrix $M_{[1,k]}$ is non-derogatory. Hence
 $\eus P\cap \pi_{G'}^{-1}(Y)=\{0\}$, and we are done.
\end{proof}
\end{ex}
\begin{rema} If we work with $\tilde G=SL(\BV)$ in place of $GL(\BV)$, then a generic stabiliser
becomes $\te_{n-1}$. Here the constant morphism $F_0$ should be omitted and the matrices 
$\eus M^{ki}$, $i\ge 1$, should be replaced with their projections to $\mathfrak{sl}(\BV)$.
\end{rema}
\noindent Thus, {\sl by Remark~\ref{rmk:about-ss} and Theorem~\ref{thm:gl-nk2}, the proof of
Theorems~\ref{thm:main1} and \ref{thm:main2} can be adjusted to the present case. 
Therefore, $\Ker(\phi)$ (resp. $\Ker(\phi_G)$) is a free $\bbk[V]$-module (resp. $\bbk[V]^G$-module) with basis $F_0,F_1,\dots,F_{n-1}$ and 
$\q=(\prod_{i=1}^k \mathfrak{gl}(\BU_i))\ltimes \bigl(\bigoplus_{i=1}^k \BU_i^*\otimes\BU_{i+1}\bigr)$ satisfies the Kostant criterion.}

It is worth noting that the special case of the involutive automorphism $\vartheta$ (i.e., if $k=2$) has already been settled in \cite[Sect.\,5]{coadj}.

\section{Constructing covariants for semi-direct products, II} 
\label{sect:primery2}

\noindent
\begin{ex}   \label{ex:sl-2-slag*}
$G=SL_n=SL(\BU)$, $V=\vp_1^2+ \vp_2^*
=\gS ^2(\BU)\oplus \wedge^2(\BU^*)$.
\\ \indent
We regard $V$ as the space of pairs of matrices: $V=\{(A,B)\mid A^t=A \ \& \ B^t=-B\}$, where the action of $g\in G$ is given by
\beq  \label{eq:pary-matriz1}
   g{\cdot}(A,B)=(gAg^t, (g^t)^{-1}Bg^{-1}) .
\eeq
and the corresponding action of $s\in\g=\sln$ is
\beq  \label{eq:pary-matriz2}
   s{\cdot}(A,B)=(sA+As^t, -s^tB-Bs) .
\eeq
In what follows, one has to distinguish the cases of even or odd $n$.
The algebra $\bbk[V]^G$ is (bi)graded polynomial and the (bi)degrees of the basic invariants 
are~\cite{ag79,gerry1}:

$\begin{cases}  (2,2), (4,4),\dots, (n-2,n-2), (n,0), (0, n/2),  & \text{ if } \ n=2k , \\
(2,2), (4,4),\dots, (n-1,n-1), (n,0) & \text{ if } \ n=2k+1 .
\end{cases}$
Here  the invariant of degree $(n,0)$ is $\det A$, and the invariant of degree $(0,n/2)$ is
$\text{Pf}\, B$. While the invariants of degree $(2i,2i)$ are just $\tr(AB)^{2i}$, $2i< n$.

A generic isotropy group is $H\simeq \BT_{[n/2]}$~\cite[Table\,2]{alela1}.
For instance, one can take
 
 $H=\begin{cases}
 \mathsf{diag}(t_1,\dots,t_k,    t_1^{-1},\dots, t_k^{-1})\mid t_i\in \bbk^\times \}, & \text{ if } \ n=2k \\
 \mathsf{diag}(t_1,\dots,t_k,1, t_1^{-1},\dots, t_k^{-1})\mid t_i\in \bbk^\times \}, & \text{ if } \ n=2k+1.
 \end{cases} $

\noindent We have to construct $[n/2]$ morphisms $\{F_i\}$ in $\Ker(\phi)$. To begin with, take
$F_1(A,B)=AB$. Since $(AB)^t=-BA$, we have $\tr(AB)=0$, and it follows from \eqref{eq:pary-matriz1} that  $g{\cdot}AB=g(AB)g^{-1}$. Hence $F_1\in\Mor_G(V,\g)$. We continue  by letting
$F_i(A,B)=(AB)^{2i-1}$, $i=1,2,\dots,[n/2]$. To ensure that the resulting matrix is traceless, we must consider only the odd powers of $AB$. Using \eqref{eq:pary-matriz2}, one verifies that
$F_i(A,B){\cdot}(A,B)=0$, hence $F_i\in \Ker (\phi)$. 
The corresponding $Q$-invariants in $\bbk[\q^*]$ are
$\hat F_i(\xi,A,B)=\tr(\xi(AB)^{2i-1})$.   Let $I_k(\bar d)$ denote the diagonal $k$ by $k$ matrix with diagonal entries $\bar d=(d_1,\dots,d_k)$. Taking
$A=\begin{pmatrix} 0_k & I_k(\bar d) \\ I_k(\bar d) & 0_k \end{pmatrix}$ and 
$B=\begin{pmatrix} 0_k & I_k(\bar c) \\ -I_k(\bar c) & 0_k \end{pmatrix}$ shows that the matrices
$(AB)^{2i-1}$, $1\le i\le k$, are linearly independent whenever the elements $\{c_jd_j\}$ are different.
Hence $F_1,\dots,F_k$ are linearly independent for $n=2k$. This construction can easily be adjusted to 
$n=2k+1$.

Having the degrees of all basic invariants and covariants, one verifies that 
$\sum_{i=1}^{[n/2]} \deg F_i+ q(V\md G)=\dim V$ if $n$ is odd; while for $n$ even
one obtains $\sum_{i=1}^{[n/2]} \deg F_i=\dim V- q(V\md G)+ [n/2]$.
Since the \ctc holds here (Example~\ref{ex:kosik1}($2^0$)), we have 
\\ \indent
{\sl if $n$ is odd, then the assumptions of Theorems~\ref{thm:main1} and \ref{thm:main2} 
are satisfied. Therefore, $\Ker(\phi)$ (resp. $\Ker(\phi_G)$) is a free $\bbk[V]$-module (resp. $\bbk[V]^G$-module) with basis $F_1,F_2,\dots,F_{[n/2]}$ and $\q:=\sln\ltimes (\vp_1^2+\vp_2^*)^*$ satisfies the Kostant criterion.}

{If $n$ is even, then the same conclusion is still true, but one have to modify the constructed covariants 
$\{F_i\}$ in order to obtain a new family such that Equality~\eqref{eq:summa} to be satisfied. This will appear in a forthcoming paper.}
\end{ex}

\begin{ex}   \label{ex:sl-2-slag}
If we slightly change $V$ of Example~\ref{ex:sl-2-slag}, i.e., take $G=SL_n$ and 
$\tilde V=\vp_1^2+\vp_2=\gS ^2(\BU)\oplus \wedge^2(\BU)$, 
then the action $(G:\tilde V)$  has similar properties. Namely, $\bbk[\tilde V]^G$ is polynomial~\cite{ag79,gerry1} 
and $\gig(G:\tilde V)=\BT_{[n/2]}$~\cite[Table\,2]{alela1}. However, the 
construction of covariants in $\Ker(\phi)$ becomes totally different and more involved.
We regard $\tilde V$ as the space of pairs of matrices: $\tilde V=\{(A,B)\mid A^t=A \ \& \ B^t=-B\}$, where the action of $g\in G$ is given by
\beq  \label{eq:pary-matriz3}
   g{\cdot}(A,B)=(gAg^t, gBg^t) .
\eeq
and the corresponding action of $s\in\g=\sln$ is
\beq  \label{eq:pary-matriz4}
   s{\cdot}(A,B)=(sA+As^t, sB+Bs^t) .
\eeq
Consider the ``characteristic polynomial''
\[
   \eus F(\lb)=\det(A+\lb B)=\sum_{i=0}^n f_i(A,B)\lb^i .
\]
Since $(A+\lb B)^t=A-\lb B$, we have $\eus F(\lb)=\eus F(-\lb)$, i.e., $\eus F(\lb)=P(\lb^2)$ and 
$f_i(A,B)\equiv 0$ unless $i$ is even.  If $n$ is odd, then $\bbk[\tilde V]^G$ is freely generated by the 
$f_{2i}$'s. For $n$ even, the only correction is that $f_{n}(A,B)=\det B$ should be replaced with $\text{Pf}\, B$~\cite{ag79,gerry1}. Therefore,
$\dim\tilde V-q(\tilde V\md G)=\begin{cases} 
2k^2-k, & \text{ if }\ n=2k 
\\ 2k^2+k, & \text{ if }\ n=2k+1
\end{cases}$. \par\noindent 
We provide below a construction of the required covariants in $\Ker(\phi)$. As in Example~\ref{ex:tri-SL}, let $A^*$ be the adjugate of $A$. Consider the morphism $F:\tilde V\to \gln$, 
$F(A,B)=BA^*$.  
\begin{lme}   \label{lm:1-for-ex-tilda-V}  We have
\begin{itemize}
\item[\sf (a)]   \ $\tr(BA^*)=0$, i.e., $F(A,B)\in\g=\sln$; 
\item[\sf (b)]   \ $F$ is $G$-equivariant;
\item[\sf (c)]   \ $F\in \Ker(\phi)$.
\end{itemize}
\end{lme}
\begin{proof}   (a) Since $A^t=A$, we have $(A^*)^t=A^*$. Hence $(BA^*)^t=-A^*B$. 
\\ \indent
(b) By definition, $F(g{\cdot}(A,B))=gBg^t\,(gAg^t)^*$. If $\det A\ne 0$, then the RHS equals
\[
   \det(gAg^t){\cdot}gBg^t{\cdot}(gAg^t)^{-1}=\det A{\cdot}gBA^{-1}g^{-1}=gBA^*g^{-1} .
\] 
Hence $F$ is a $G$-equivariant mapping from $\tilde V$ to $\g=\sln$ on the dense open subset of 
$\tilde V$, where $A$ is invertible. Since $F$ is a polynomial morphism, this holds on the whole of $\tilde V$.
\\ \indent
(c)  We have $F(A,B){\cdot}(A,B)=(BA^*A+A(BA^*)^t, BA^*B+B(BA^*)^t)=$ \\
\phantom{a} \hfill $(BA^*A-AA^*B, BA^*B-BA^*B)=0$.
\end{proof}
Having constructed one suitable covariant, we perform a ``polarisation''. Consider
\[
   \mathbb H_\lb(A,B):=F(A+\lb B, B)=B(A+\lb B)^*=\sum_{i=0}^{n-1} F_i(A,B)\lb^i .
\]
Clearly $F_0(A,B)=F(A,B)$ and if $n$ is odd, then $\det B=0$ and the coefficient of $\lb^{n-1}$ 
equals $BB^*=(\det B)I=0$.

\begin{thme}   \label{thm:sl-2-slag*} We have
\begin{itemize}
\item[\sf (a)]  \ $\tr F_{2i}(A,B)=0$ for all $i$;
\item[\sf (b)]  \ $\mathbb H_\lb$ is a $G$-equivariant mapping from $\tilde V$ to $\gln$. In particular, 
$F_{2i}\in \Mor_G(\tilde V,\g)$ for all $i$; 
\item[\sf (c)]  \ $F_{2i}\in \Ker(\phi)$ for all $i$.
\end{itemize}
\end{thme}
\begin{proof}
(a) If both $A$ and $A+\lb B$ are invertible, then 
\begin{multline*}
B(A+\lb B)^*=\det(A+\lb B){\cdot}B(A+\lb B)^{-1} \\
=\det(A+\lb B){\cdot}B\bigl(A(I+\lb A^{-1}B)\bigr)^{-1}= 
\det(A+\lb B){\cdot}B(I+\lb A^{-1}B)^{-1}A^{-1}= \\
\det(A+\lb B){\cdot}\bigl(BA^{-1}+\lb (BA^{-1})^2+\lb^2(BA^{-1})^3+\dots\bigr)
\end{multline*}
Since $A$ is symmetric, so is $A^{-1}$ and therefore $(BA^{-1})^{2i+1}$ is a product of a symmetric and a 
skew-symmetric matrices. Hence $\tr(BA^{-1})^{2i+1}=0$. As $\det(A+\lb B)=P(\lb^2)$, the total 
coefficient of $\lb^{2i}$ is a traceless matrix. Since this is true for a dense open subset of triples
$(A,B,\lb)$ such that $A$ and $A+\lb B$ are invertible, and $\mathbb H_\lb$ is a polynomial mapping, this holds for all triples.
\\ \indent
(b) The proof is the same as in Lemma~\ref{lm:1-for-ex-tilda-V}(b).
\\ \indent
(c)   We prove that $\mathbb H_\lb^{even}:=\sum_{i}F_{2i}\lb^{2i}\in \Ker(\phi)$ for all $\lb$. Equivalently, 
only odd powers of $\lb$ survive in $\mathbb H_\lb(A,B){\cdot}(A,B)$. By definition, 
\[
  \mathbb H_\lb(A,B){\cdot}(A,B)=(B(A{+}\lb B)^*A{+}A(B(A{+}\lb B)^*)^t, B(A{+}\lb B)^*B{+}B(B(A{+}\lb B)^*)^t).
\]
Let us transform the first component in the RHS. Again, assuming first that $A$ and $A+\lb B$ are invertible, one obtains:
\begin{multline*}
(\mathbf{F1})=
B(A{+}\lb B)^*A=\det(A{+}\lb B)B(A+\lb B)^{-1}A\\
=\det(A{+}\lb B)B\bigl(A(I{+}\lb A^{-1} B)\bigr)^{-1}A
=\det(A{+}\lb B)B(I{+}\lb A^{-1} B)^{-1}\\
=\det(A{+}\lb B)(B{-}\lb(BA^{-1}B){+}\lb^2(BA^{-1}BA^{-1}B)-\dots) 
\end{multline*}
and
\begin{multline*}  (\mathbf{F2})=
A(B(A{+}\lb B)^*)^t=-\det(A{+}\lb B)(A(A{-}\lb B)^{-1}B)\\
=-\det(A{+}\lb B)A\bigl((I{-}\lb BA^{-1})A\bigr)^{-1}B 
=-\det(A{+}\lb B)(I{-}\lb BA^{-1})^{-1}B \\
=-\det(A{+}\lb B)(B{+}\lb(BA^{-1}B){+}\lb^2(BA^{-1}BA^{-1}B)+\dots)
\end{multline*}
Because $\det(A{+}\lb B)=P(\lb^2)$, the sum $(\mathbf{F1}){+}(\mathbf{F2})$ contains only odd powers of 
$\lb$. Again, using the polynomiality of $\mathbb H_\lb$, we conclude that this property holds for all 
$A,B,\lb$.
\\  \indent
The argument for the second component is similar.
\end{proof}
\noindent
Thus, we have constructed $[n/2]$ covariants $F_{2i}$ ($0\le 2i \le n-2$) in $\Ker(\phi)$. These 
covariants are linearly independent, because their bi-degrees are different. Since $\deg F_{2i}=n$ 
for all $i$, we have
$\sum_{i=0}^{[n/2]-1}\deg F_{2i}-\dim V+q(V\md G)=0$ if $n$ is odd (and $=[n/2]$ if $n$ is even).
Since the \ctc holds here (Example~\ref{ex:kosik1}($1^o$)), we conclude that 
\\ \indent
{\sl if $n$ is odd, then the assumptions of Theorems~\ref{thm:main1} and \ref{thm:main2} are satisfied. Therefore, $\Ker(\phi)$ (resp. $\Ker(\phi_G)$) is a free $\bbk[V]$-module (resp. $\bbk[V]^G$-module) 
with basis $F_0,F_2,\dots,F_{2[n/2]-2}$ and $\q:=\sln\ltimes (\vp_1^2+\vp_2)^*$ satisfies the Kostant 
criterion.}

{If $n$ is even, then the same conclusion is still true, but one have to modify the constructed covariants 
$\{F_i\}$ in order to obtain a new family such that Equality~\eqref{eq:summa} to be satisfied. This will 
appear in a forthcoming paper.}
\end{ex}

\begin{ex}  \label{ex:sp-so}
$G=Sp(\BU)\times SO(\BV)$, $V=\BU\otimes \BV$. \\ 
\indent
This representation is a $\vartheta$-group associated with an outer automorphism of 
order $4$ of $\mathfrak{sl}(\BU\oplus\BV)$. Therefore  $\bbk[V]^G$ is polynomial and
$\N_G(V)$ contains finitely many $G$-orbits, cf. Example~\ref{theta-1}. Furthermore, a generic stabiliser is 
reductive if and only if either $\dim \BU\ge \dim \BV$ or $\dim\BV-\dim\BU\in 2\BN$. In these cases, the 
action is stable and hence \ctc holds.

Set $2m=\dim \BU$ and $n=\dim \BV$.
Let $\gI$ (resp. $\gJ$) be a symmetric (resp. skew-symmetric) matrix of order $n$ (resp. $2m$) such 
that $\gI^2=I$ (resp. $\gJ^2=-I$).
We regard $SO(\BV)$ (resp. $Sp(\BU)$) as the group that preserves the bilinear form with matrix $\gI$ (resp. $\gJ$). Then 
\begin{equation}   \label{eq:so-sp}
\begin{array}{c} 
\mathfrak{sp}(\BU)=\mathfrak{sp}_{2m}=\{X\in \mathsf{Mat}_{2m} \mid 
X^t\gJ+\gJ X=0\}=\{X \mid (\gJ X)^t=\gJ X\};   \\[.6ex]
\mathfrak{so}(\BV)=\mathfrak{so}_{n}=\{Y\in \mathsf{Mat}_n 
\mid Y^t\gI+\gI Y=0\}=\{Y \mid (\gI Y)^t=-\gI Y\} .
\end{array}
\end{equation}
We identify $\BU\otimes \BV$ with the space of $2m$ by $n$ matrices, where the action of 
$\mathfrak{sp}(\BU)\times \mathfrak{so}(\BV)$  is 
given by $(s_1,s_2){\cdot}M =s_1M -M s_2$. Here a generic isotropy group is a torus if and only if $0\le \dim \BV-\dim \BU\le 2$ (more precisely, if $n\ge 2m$, then $\gig(G:V)\simeq SO_{n-2m}\times \BT_m$).
The corresponding possibilities are considered below.
(Whenever it is convenient, we may assume that $\gI=I$; and then $\mathfrak{so}(\BV)$ consists of the usual skew-symmetric matrices.)

(i)  
\un{Assume that $\dim \BU=\dim \BV=2m$}. Here $\gig(G:V)=\BT_m$ and this torus is embedded diagonally in $Sp(\BU)\times SO(\BV)$.
The degrees of the basic invariants of $\bbk[V]^G$ are 
$4,8,\dots, 4m-4, 2m$~\cite{litt89}. Hence $\dim V-q(V\md G)=4m^2-2m-2m(m-1)=2m^2$.

Define the covariant $F_1: \BU\otimes \BV  \to \mathsf{Mat}_n\times \mathsf{Mat}_n$
by $F_1(M )=(M\gI M^t \gJ, \gI M ^t \gJ M )$. Using Eq.~\eqref{eq:so-sp}, one verifies that 
$F_1(M )\in \g=\mathfrak{sp}(\BU)\times \mathfrak{so}(\BV)$. Moreover,  
$F_1$ is $G$-equivariant, and 
$F_1(M ){\cdot}M =0$, i.e., $F_1\in \Ker(\phi)$. 
If a matrix $R$ is either symplectic or orthogonal, then so is $R^{2i-1}$ for any $i$. Therefore, the 
covariants 
\[
   F_i(M)=F_1(M)^{2i-1}=\bigl( (M\gI M^t \gJ)^{2i-1}, (\gI M ^t \gJ M)^{2i-1}\bigr)  .
\] 
are well-defined. Moreover, $F_1,\dots,F_m$ are linearly 
independent. (Assuming for simplicity that $\gI=I$, one easily verifies that $F_1(D),\dots, F_m(D)$ are 
linearly independent for a generic diagonal matrix $D$.) 
Here $\deg F_i=2(2i-1)$. Hence 
$\sum_{i=1}^m\deg F_i=2m^2$, so that \eqref{eq:summa} holds.
Thus, Theorems~\ref{thm:main1} and  \ref{thm:main2} apply here. 

(ii)   
\un{Assume that $\dim \BV=2m+1=\dim \BU+1$}. Here again $\gig(G:V)=\BT_m$ 
and this torus is embedded diagonally in $Sp(\BU)\times SO(\BV)$. The degrees of the basic 
invariants of $\bbk[V]^G$ are $4,8,\dots , 4m$~\cite{litt89}. Hence 
$\dim V-q(V\md G)=2m(2m+1)-2m(m+1)=2m^2$, as in (i). The formulae for
$F_i$, $i=1,\dots,m$, also remain the same.  Note only that now $\gI$ and $\gJ$ have different order
and therefore the matrices $M\gI M^t \gJ$ and $\gI M ^t \gJ M$ are of order $2m$ and $2m+1$, 
respectively.

(iii)  
\un{Assume that $\dim \BV=2m+2=\dim \BU+2$}. Here  $\gig(G:V)=\BT_{m+1}$, but only an $m$-dimensional subtorus is embedded diagonally in $G$, whereas a complementary 
$1$-dimensional torus belongs to $SO(\BV)$. (This is not surprising, since $\rk Sp(\BU)=m$.)
The degrees of the basic invariants of $\bbk[V]^G$ are
$4,8,\dots, 4m$, as in (ii). Hence $\dim V-q(V\md G)=2m(2m+2)-m(2m+2)=2m^2+2m$.

As in (i), we construct the linearly independent covariants $F_1,\dots,F_m$ with 
$\sum_{i=1}^m\deg F_i=2m^2$, but this is not 
sufficient now. These $m$ covariants take a generic $G$-regular element $M\in V$ to the diagonally 
embedded $m$-dimensional torus in the stabiliser $\g_M\simeq \te_{m+1}$.
We need one more covariant (of degree $2m$) that takes $M$ to a 1-dimensional subtorus sitting 
only in $\mathfrak{so}(\BV)$. In other words, starting with a $2m$ by $2m+2$ matrix $M$, we wish 
to construct, in a natural way, a skew-symmetric matrix of order $2m+2$. Here is the solution:
Let $M_{ij}$ be the square matrix of order $2m$ obtained by deleting the $i$-th and $j$-th columns 
from $M$, $1\le i<j\le 2m+2$. 
We then set 
\\
$a_{ij}=\begin{cases}   (-1)^{i+j}\det M_{ij}, &  \ \text{ if } \ i<j \\
0, &  \ \text{ if } \ i=j \\
-a_{ji}, &  \ \text{ if } \ i>j .
\end{cases}$.  Clearly, $A_M=(a_{ij})$ is a skew-symmetric matrix of order $2m+2$, and we define
$F_{m+1}(M)=(0, A_M)\in \mathfrak{sp}(\BU)\times \mathfrak{so}(\BV)$. It is easily seen that $F_{m+1}$ is equivariant, $\deg F_{m+1}=2m$, and 
$F_{m+1}(M){\cdot}M=(0,-MA_M)=0$. 
Thus, 
\\ \indent 
{\sl if\/ $0\le \dim\BV-\dim\BU\le 2$, then 
the assumptions of Theorems~\ref{thm:main1} and \ref{thm:main2} 
are satisfied. Therefore, $\Ker(\phi)$ (resp. $\Ker(\phi_G)$) is a free $\bbk[V]$-module (resp. $\bbk[V]^G$-module) with basis $\{F_i\}$ and
$\q=(\mathfrak{sp}(\BU)\times \mathfrak{so}(\BV))\ltimes (\BU\otimes \BV)^*$
satisfies
the Kostant criterion.}
\end{ex}  

\begin{ex}   \label{ex:so-mnogo}
$G=SO(\BV)=SO_{n+2}$ and $V=n\BV$, the sum of $n$ copies of the defining representation of 
$SO_{n+2}$. 
\\ \indent
Here $\gig(G:V)=SO_2\simeq \BT_1$,
$V\md G\simeq \mathbb A^{(n+1)n/2}$, and $q(V\md G)=(n+1)n$. 
The explicit construction of the unique covariant of degree $\dim V-q(V\md G)=n$ in 
$\Ker(\phi)$ is similar to the construction of $F_{m+1}$ in Example~\ref{ex:sp-so}(iii).
We regard an element of $V$ as $n+2$ by $n$ matrix $M$ and consider its minors of order $n$,  
$\det M_{ij}$, where $1\le i< j\le n+2$. Then $F(M)=(a_{ij})$, where $a_{ij}=(-1)^{i+j}\det M_{ij}$ for 
$i<j$, etc. 
\end{ex}

\appendix
\section{Tables of representations with toral generic stabilisers}  
\label{sect:tables}

\noindent
Using classification results of Elashvili~\cite{alela1,alela2}, one can write down the arbitrary 
representations of simple algebraic groups or the irreducible representations of semisimple groups 
whose generic stabiliser is toral. The subsequent four tables include {\bf all} such representations.
But their content is not disjoint. Recall that $\q=\g\ltimes V^*$ and we are interested in the symmetric invariants of $\q$.
\\  \indent
In Table~\ref{table-te1}, we gather all representations with $1$-dimensional generic stabiliser.
The column {\sf (FA)} (resp. ({\sf Eq)}) refers to the presence of the property that $\bbk[V]^G$ is a polynomial ring (resp.
$\pi_{V,G}:V\to V\md G$ is equidimensional). This information is inferred from tables in \cite{ag79, litt89, gerry1, gerry2}.
Results of Section~\ref{sect:codim2} imply that \ctc holds for all these representations. By Theorem~\ref{thm:l=1}, $\bbk[\q^*]^{R_u(Q)}$ is a polynomial ring for all items of this table. Furthermore, if ({\sf FA})
holds, then $\bbk[\q^*]^{Q}$ is a polynomial ring. Finally, if ({\sf Eq}) holds, then $\eus U(\q)$ is a free module over $\eus Z(\q)$.
\begin{table}[h]
\caption{Representations $(G:V)$ with $\mathsf{g.s.}(\g:V)=\te_1$}   \label{table-te1}
\begin{center}
\begin{tabular}{>{\sf}c<{}|>{$}c<{$}>{$}c<{$}>{$}c<{$}>{$}c<{$}>{$}c<{$}cc}
 {\rus N0}&  G & V & \dim V & \dim V\md G & q(V\md G) &   ({\sf FA}) & ({\sf Eq})    \\ \hline \hline
1a & SO_{n+2}{\times} SO_{n} & \vp_1\vp'_1 & n(n{+}2) & n & n(n{+}1) &  {\it yes} &  {\it yes}\\ 
1b & SO_{n+2} & n\vp_1 & n(n{+}2) & \genfrac{(}{)}{0pt}{}{n+1}{2} & n(n{+}1) &  {\it yes} & {\it yes}, 
if $n{\le} 3$ \\[.5ex] \hline
2 & SL_6 & 2\vp_2{+}\vp^*_2 & 45 & 11 & 36 &  {\it yes} &  {no} \\ \hline
3a & SL_6{\times} SL_3 & \vp_2\vp'_1 & 45 & 3 & 36 &  {\it yes} & {\it yes} \\  
3b & SL_6{\times} SL_2 & \vp_2\vp'^2 & 45 & 8 & 36 &  no & no \\
3c & SL_6 & 3\vp_2 & 45 & 11 & 36 &  {\it yes} &  {no}\\ \hline
4a & Sp_6{\times} SL_2 & \vp_2\vp' & 28 & 5 & 22 &  no & no \\
4b & Sp_6 & 2\vp_2 & 28 & 8 & 22 &  {\it yes} & {\it yes}\\ \hline
5 & SL_4{\times} SL_2 & \vp_2\vp'^3 & 24 & 7 & 20  &  no & no\\  
\hline
\end{tabular}
\end{center}
\end{table}

{\small {\bf Remarks.}  1) In {\rus N0}\,1(a,b), we have $V\md SO_{n+2}\simeq S^2\vp_1'$ as $SO_n$-module.

2) \ In {\rus N0}\,3(a,b,c), we have $V\md SL_6\simeq \sfr_6+\sfr_2+\sfr_0$ as $SL_2$-module.

3) \ In {\rus N0}\,4(a,b), we have $V\md Sp_6\simeq \sfr_3+\sfr_2+\sfr_0$ as $SL_2$-module.

4) \ In {\rus N0}\,5, we have $V\md SL_4\simeq \wedge^2(\sfr_3)=\sfr_6+\sfr_2$ as $SL_2$-module.
\\
This explains why $\bbk[V]^G$ is not polynomial in 3b, 4a, 5.}


\begin{table}[h]
\caption{$\vartheta$-groups with toral generic stabiliser and their ``restrictions''}   \label{table-2}
\begin{center}
\begin{tabular}{c | >{$}c<{$} >{$}c<{$} >{$}c<{$} >{$}c<{$} c} 
 {\rus N0} &  G_0 & \g_1 & \h & \vartheta  & Ref.\\  \hline\hline 
1 & SO_{2m}{\times} Sp_{2m} & \vp_1\vp_1' & \te_m & (\GR{A}{4m-1}^{(2)}, 4) & Example~\ref{ex:sp-so}(i) \\[1ex]
2 & SO_{2m+1}{\times} Sp_{2m} & \vp_1\vp_1' & \te_m & (\GR{A}{4m}^{(2)}, 4) & Example~\ref{ex:sp-so}(ii)\\[1ex]
3 & SO_{2m+2}{\times} Sp_{2m} & \vp_1\vp_1' & \te_{m+1} 
&(\GR{A}{4m+1}^{(2)},4) & Example~\ref{ex:sp-so}(iii) \\[1ex]
4 & (SL_n)^k\times \BT_{k-1} & \displaystyle\sum_{i=1}^n \vp_1^{(i)}\vp_{n-1}^{(i+1)} & \te_{n-1} & (\GR{A}{kn-1}, k) &
Example~\ref{gl-nk} \\[.5ex]
5 & SL_4{\times} SL_4{\times} SL_2 & \vp_1\vp_1'\vp'' & \te_3 & (\GR{E}{7}, 4) & Example~\ref{ex:tri-SL} \\[1ex]
6 & SL_6{\times} SL_3 & \vp_2\vp'_1 &  \te_1 & (\GR{E}{7}, 3) &  Example~\ref{ex:perexod} \\[.5ex]  
6a & SL_6 & 3\vp_2 &  \te_1 & - &  Example~\ref{ex:perexod} \\[.5ex]  
7 & SL_6{\times} SL_2 & \vp_3\vp' & \te_2 & (\GR{E}{6}, 2) & \cite[Sect.\,5]{coadj} \\ 
7a & SL_6 & 2\vp_3 & \te_2 & - & \\  
8 & SO_{n+2}{\times} SO_{n} & \vp_1\vp'_1 & \te_1 & (\GR{D}{n+1}, 2) & \cite[Sect.\,5]{coadj} \\  
8a & SO_{n+2} & n\vp_1 &  \te_1 & - & Example~\ref{ex:so-mnogo} \\
\hline
\end{tabular}
\end{center}
\end{table}

In Table~\ref{table-2}, we gather all representations with a toral generic stabiliser that are 
$\vartheta$-groups in the sense of Vinberg (Example~\ref{theta-1}) and related restrictions.
Namely, items 6a, 7a, and 8a (which are not  $\vartheta$-groups!)  are obtained from the genuine 
$\vartheta$-groups (items 6,7,8) by omitting the second factor of $G_0$ and we say that these are 
"restrictions" of $\vartheta$-groups. It appears that this passage does not change generic isotropy groups, 
which are always contained in the first factor of $G_0$. Moreover, the number $q(\g_1\md G_0)$ is not 
affected, too. Therefore the the covariants $\{F_i\}$ produced for these $\vartheta$-groups, as described in
the respective examples,  are also 
suitable for their ``restrictions''.
The symbol $(\GR{X}{n}^{(k)}, m)$ in column ``$\vartheta$'' represents the following information on the 
automorphism $\vartheta$ of $\g$. Here $\GR{X}{n}$ is the Cartan type of $\g$, $m$ is the order of 
$\vartheta$, and $k$ is the minimal integer such that $\vartheta^k$ is inner (this number is omitted if it 
equals $1$).
\begin{table}[h]
\caption{Reduced $\vartheta$-groups  with toral generic stabilisers}   \label{table-3}
\begin{center}
\begin{tabular}{c | >{$}c<{$} >{$}c<{$} >{$}c<{$} |>{$}c<{$} |c} 
 {\rus N0} &  G(0)' & \g(1) & \h & \BZ\text{-grading of}\ \g  & Ref.\\  \hline\hline 
1 & SL_2{\times} SL_2{\times} SL_2 & \vp\vp'\vp'' & \te_2 & (\GR{D}{4}, \ap_2) 
\begin{picture}(75,23)(0,7)
\setlength{\unitlength}{0.015in}
\put(35,18){\circle*{5}} 
\put(35,3){\circle{5}}
\multiput(20,18)(15,0){3}{\circle{5}}
\multiput(23,18)(15,0){2}{\line(1,0){9}}
\put(35,6){\line(0,1){9}}
\end{picture}
&
Example~\ref{ex:cubic-matr} \\
2 & SL_3{\times} SL_3{\times} SL_2 & \vp_1\vp'_1\vp'' & \te_2 & (\GR{E}{6}, \ap_3) 
\begin{picture}(75,28)(0,7)
\setlength{\unitlength}{0.015in}
\put(35,18){\circle*{5}} 
\put(35,3){\circle{5}}
\multiput(5,18)(15,0){5}{\circle{5}}
\multiput(8,18)(15,0){4}{\line(1,0){9}}
\put(35,6){\line(0,1){9}}
\end{picture}
& 
Example~\ref{ex:tri-SL} \\[.5ex]  \hline
\end{tabular}
\end{center}
\end{table}

\noindent
The theory of $\vartheta$-groups implies that, for all items of Tables~\ref{table-2} and \ref{table-3}, 
$\bbk[V]^G$ is a polynomial ring. Our description of the corresponding covariants shows that, for all items
except {\rus N0}5 in Table~\ref{table-2} and {\rus N0}2 in Table~\ref{table-3},
$\bbk[\q^*]^Q$ is a polynomial ring, the Kostant criterion is always satisfied for $\q$, and $\Ker(\phi)$ is 
a free $\bbk[V]$-module generated by $G$-equivariant morphisms. The explicit construction of covariants 
$F_i\in\Ker(\phi)$ is given in the examples mentioned in the column ``Ref.'' 
\begin{table}[h]
\caption{The remaining representations with toral generic stabiliser}   \label{table-4}
\begin{center}
\begin{tabular}{ >{\sf}c<{} | >{$}c<{$} >{$}c<{$} >{$}c<{$} cc c} 
 {\rus N0} &  G & V & \h & Rem.  & ({\sf FA}) & Ref.\\  \hline\hline 
1 & SL_n &  \vp_1^2+\vp_2^* & \te_{[n/2]} & $n\ge 4$ 
& {\it yes} & Example~\ref{ex:sl-2-slag*} \\[.5ex]
2 & SL_n &  \vp_1^2+\vp_2 & \te_{[n/2]} &  $n\ge 4$ & {\it yes} & Example~\ref{ex:sl-2-slag}\\ 
3 & SL_n{\times} SL_n{\times} SL_2 & \vp_1\vp_1'\vp'' & \te_{n-1} & $n\ge 5$ & no & Example~\ref{ex:tri-SL}  \\ 
3a & SL_n{\times} SL_n & \vp_1\vp_1'+\vp_1\vp'_1 & \te_{n-1} & $n\ge 3$ & {\it yes} & Example~\ref{ex:tri-SL}  \\ 
4 & SL_8 &  \vp_3+\vp_7 &  \te_2 & - 
& {\it yes} & \\
5 & SL_8 &  \vp_3+\vp_1 &  \te_2 & - & {\it yes} & \\
6 & Sp_4{\times} SO_7 & \vp_1\vp'_3 & \te_2 & - &  {\it yes} & \\
\hline
\end{tabular}
\end{center}
\end{table}

\noindent
The column {\sf (FA)} in Table~\ref{table-4} refers to the presence of the property that $\bbk[V]^G$ is a 
polynomial ring. For items {\sf 1,2,3a}, $\bbk[\q^*]^Q$ is a polynomial ring, while in case {\sf 3}, only 
$\bbk[\q^*]^{R_u(Q)}$ is a polynomial ring. (A more precise information can be found in the respective examples.)
\\ \indent
We do not know whether it is possible to construct  covariants $F_1,F_2\in \Ker(\phi)$ for 
items 4-6 of Table~\ref{table-4} such that $\deg F_1+\deg F_2=\dim V -q(V\md G)$ and whether the $\bbk[V]$-module $\Ker(\phi)$ is free or $\bbk[\q^*]^Q$ is 
a polynomial ring in these cases. Nevertheless, using Theorem\,2.8 in \cite{Y16}, Remark~\ref{rem:iz-kota}, 
and the fact that $H=\gig(G:V)\simeq \BT_2$ is connected, one can prove that there do exist certain linearly 
independent $G$-equivariant morphisms $F_1, F_2\in \Ker(\phi)$. However, this existence assertion says
nothing about their degrees.


\end{document}